\documentclass[12pt]{amsart}
\usepackage{amssymb,latexsym,amsmath,enumerate,nicefrac}
\usepackage{amsaddr}

\usepackage{color}
\definecolor{myurlcolor}{rgb}{0.6,0,0}
\definecolor{mycitecolor}{rgb}{0,0,0.8}
\definecolor{myrefcolor}{rgb}{0,0,0.8}
\usepackage[bookmarks=false,pagebackref=true]{hyperref}
\hypersetup{colorlinks,
linkcolor=myrefcolor,
citecolor=mycitecolor,
urlcolor=myurlcolor}

\allowdisplaybreaks

\usepackage{dsfont}
\usepackage{mathrsfs}

\newcommand{\dsp}{\hspace{3em}}
\newcommand{\www}{.95}
\newcommand{\hhh}{.95}
\newcommand{\dwscale}{1.06}
\usepackage{tikzpagenodes}
\newcounter{dwfig}
\usepackage{wrapfig}
\newcommand{\C}{{\mathscr C}}
\renewcommand{\P}{{\mathcal P}}
\newcommand{\A}{{\mathscr A}}

\newtheorem{theorem}{Theorem}
\newtheorem{lemma}{Lemma}

\newtheorem{proposition}{Proposition}
\newtheorem*{corollary}{Corollary}
\theoremstyle{definition}
\newtheorem{definition}{Definition}

\theoremstyle{remark}

\newtheorem*{remark}{Remark}

\newcommand{\lpar}{{\textnormal(}}
\newcommand{\rpar}{\!{\textnormal)}}

\usepackage{xcolor}
\usepackage{color}
\usepackage{pgf,tikz}
\usepackage{pgfplots}
\usetikzlibrary{calc}
\usetikzlibrary{positioning}
\usepackage{environ}
\usepackage{scrextend}
\usetikzlibrary{shadows}
\usetikzlibrary{fadings}

\usepackage[T1]{fontenc}

\usetikzlibrary{backgrounds}
\usetikzlibrary{decorations}
\date{}

\tikzfading[name=fade out, inner color=transparent!0,
  outer color=transparent!100]

\usetikzlibrary{shapes,chains,scopes,positioning,arrows}
\tikzset{block/.style={text centered}}

\newcommand{\dwcred}{red}
\newcommand{\dwcorange}{orange}

\newcommand{\dwcgreen}{green}
\newcommand{\dwcblue}{blue}
\newcommand{\dwcviolet}{violet}
\newcommand{\dwcpurple}{purple}

\newcommand{\dwcteal}{teal}
\newcommand{\dwccyan}{cyan}

\newcommand{\pipice}{

\pgfmathsetmacro{\hhh}{\dwscale*2.38}

 \pgfmathsetmacro{\l}{\linewidth/28.440}%
 \pgfmathsetmacro{\h}{\l/2.65}%

 \pgfmathsetmacro{\ww}{\www/2}%
 \pgfmathsetmacro{\hh}{\hhh/2 + .015}%
 \pgfmathsetmacro{\rrr}{\hhh/4.03}%

\begin{scope}[shift = {(.766in,-.979in)}]

\begin{scope}[rotate = 0]

\begin{scope}[shift = {(\ww in,\hh in)}]

\draw[line width = .75pt, \dwcblue] (0,0) circle (\rrr in);

\draw[] (30: \dwscale*3)--(90:\rrr in);
\draw[] (-30: \rrr in)--(30:\dwscale*3);
\draw[] (0,0) --(90: \rrr in);
\draw[] (0,0) --(-30: \rrr in);

\pgfmathsetmacro{\y}{\rrr/1.735}
\pgfmathsetmacro{\z}{-\y}

 \foreach \d in {30}
 {
\begin{scope}[shift = {(\d:\rrr in)}] 
\draw[rotate = \d] (0, \z in) -- (0, \y in);
\end{scope}
}

\foreach \d in {90, 30, 330}
{
\draw[fill = black] (\d:\rrr in) circle (.25ex);
}

\draw[] (0,0) -- (30: 2*\rrr in);

\draw[] (90:\rrr in) -- (-30: \rrr in);

\draw[\dwccyan] (0, 0) -- (60: 2*\y in);

\draw[\dwccyan] (0,0) -- (0: 2*\y in);

\draw[fill = black] (0,0) circle (.4ex);

\begin{scope}[shift = {(30: \rrr in)}]
\draw[rotate = -60] (0,0) rectangle (45: 2.5ex);
\end{scope}

\begin{scope}[shift = {(30: .5*\rrr in)}]
\draw[rotate = -60] (0,0) rectangle (45: 2.5ex);
\end{scope}

\begin{scope}[shift = {(0, 2ex)}]
\node[black,  inner sep = 0pt] at (30:2*\rrr in) {$A$};
\end{scope}

\begin{scope}[rotate = 30]
\foreach \d in {30, 330}
{ 
\draw[black, fill = black] (\d:1.154*\rrr in) circle (.25ex);
}
\end{scope}

\draw[black, fill = black] (30:2*\rrr in) circle (.25ex);
\draw[black, fill = black] (30:.5*\rrr in) circle (.25ex);

\begin{scope}[shift ={(0,2ex)}]
\node[black,  inner sep = 0pt] at (60:1.154*\rrr in) {$D$};
\end{scope}

\begin{scope}[shift ={(-30:2ex)}]
\node[black,  inner sep = 0pt] at (360:1.154*\rrr in) {$E$};
\end{scope}

\node[shift = {(0, 2ex)}, black,  inner sep = 0pt] at (90: \rrr in) {$B$};

\node[shift = {(0, -2ex)}, black,  inner sep = 0pt] at (0: 0) {$O$};

\node[inner sep = .1pt, shift ={(-30:2ex)}, black, fill = white] at (-30: \rrr in) {$C$};

\node[shift = {(70: 2.5ex)}, black,  inner sep = 0pt] at (30: \rrr in) {$M$};

\node[shift = {(70: 2.5ex)}, black,  inner sep = 0pt] at (30: .5*\rrr in) {$F$};

\end{scope}

\end{scope}

\end{scope}

}

\newcommand{\pipicf}{

\pgfmathsetmacro{\hhh}{\dwscale*2.38}

 \pgfmathsetmacro{\l}{\linewidth/28.440}%
 \pgfmathsetmacro{\h}{\l/2.65}%

 \pgfmathsetmacro{\ww}{\www/2}%
 \pgfmathsetmacro{\hh}{\hhh/2}%
 \pgfmathsetmacro{\rrr}{\hhh/2.1}%
\begin{scope}[shift = {(.1in, -1.142in)}]

\begin{scope}[shift = {(\ww in,\hh in)}]

\draw[line width = .75pt, \dwcblue] (20: \rrr in) arc (20:190:\rrr in);

\pgfmathsetmacro{\newr}{1/cos(60)}
\pgfmathsetmacro{\newrr}{1/cos(20)}
\pgfmathsetmacro{\hyp}{\rrr*\newr}
\pgfmathsetmacro{\nhyp}{\rrr*\newrr}
\pgfmathsetmacro{\a}{1/cos(30)}
\pgfmathsetmacro{\b}{\rrr*\a}
\pgfmathsetmacro{\length}{\b - \rrr}
\pgfmathsetmacro{\c}{1/cos(15)}
\pgfmathsetmacro{\crrr}{\rrr*\c}
\pgfmathsetmacro{\d}{sin(15)}
\pgfmathsetmacro{\drrr}{\rrr*\d}

\draw[style={shorten >=-.2in}, black, rotate = 45] (0: \rrr in)--(60: \hyp in);      
\draw[black, rotate = 45] (0: \rrr in)--(340:\nhyp in);

\draw[black, rotate = 45] (120: \rrr in)--(0: 0);  

\draw[style={shorten >=-.2in}, black, rotate = 45] (120: \rrr in)--(60: \hyp in);
\draw[black, rotate = 45] (120: \rrr in)--(140:\nhyp in);

\draw[black, rotate = 45] (0, 0)--(60: \hyp in);

\draw[line width = .75pt, \dwcred, rotate = 45] (60: \rrr in)--(60: \hyp in);

\draw[black, rotate = 0] (105: \rrr in)--(135: \b in); 

\draw[black, rotate = 45] (0,0) -- (90: \b in);

\draw[\dwcred, line width = .75pt, rotate = 45] (90:\rrr in) -- (90: \b in);

\draw[black, rotate = 45] (0: \rrr in)--(120: \rrr in);

\draw[black, rotate = 45] (60: \rrr in) -- (120: \rrr in);

\draw[black, rotate = 45] (120: \rrr in) -- (90: \rrr in);

\draw[black, rotate = 45] (60: \rrr in) -- (90: \rrr in);

\draw[\dwcviolet, line width = .75pt, rotate = 45] (90: \rrr in) -- (75: \crrr in);

\draw[\dwcviolet,  line width = .75pt, rotate = 45] (90: \rrr in) -- (105: \crrr in);

\draw[\dwcviolet,  line width = .75pt] (120: \crrr in) -- (105: \b in);

\begin{scope}[shift = {(120: \crrr in)}]

\draw[] (0,0) -- (135: 2*\length in);

\begin{scope}[shift = {(45:\drrr in)}]

\draw[] (0,0) -- (135: 2.97*\length in);

\begin{scope}[shift = {(45:.195*\drrr in)}]

\draw[black!30!\dwcgreen, line width = .75pt] (0,0) -- (135: .488*\rrr in);

\end{scope}

\end{scope}

\end{scope}

\draw[black, fill = black] (0, 0) circle (.4ex);  
\node[shift = {(2ex, 0)}] at (0,0) {$O$};

\draw[black, fill = black] (45: \rrr in) circle (.3ex);
\node[shift = {(2ex, 0)}] at (45: \rrr in) {$H$};

\draw[black, fill = black] (165: \rrr in) circle (.3ex);
\node[inner sep = 0pt, shift = {(-2ex, 0)}] at (165: \rrr in) {$A$};

\draw[black, fill = black] (105: \hyp in) circle (.3ex);
\node[shift = {(2ex, 0)},  inner sep = 10pt] at (105: \hyp in) {$G$};

\draw[black, fill = black, rotate = 45] (60: .5*\hyp in) circle (.3ex);
\begin{scope}[rotate = 45]
\node[shift = {(2ex, -1ex)}] at (60: .5*\hyp in)  {$N$};
\end{scope}

\begin{scope}[rotate = 45]
\draw[black, fill = black] (90: .5*\hyp in) circle (.3ex);
\node[shift = {(-1.5ex, 0)}] at (90: .5*\hyp in) {$I$};
\end{scope}

\draw[black, fill = black, rotate = 45] (90: \b in) circle (.3ex);
\begin{scope}[rotate = 45]
\node[shift = {(-2ex, 0)}] at (90: \b in) {$C$};
\end{scope}

\draw[black, fill = black] (120: \crrr in) circle (.3ex);
\node[shift = {(-2ex, .5ex)}] at (120: \crrr in) {$J$};

\draw[black, fill = black] (150: \crrr in) circle (.3ex);
\node[shift = {(-2ex, 0)}] at (150: \crrr in) {$B$};

\draw[black, fill = black] (105: \b in) circle (.3ex);
\node[shift = {(2ex, 0)}] at (105: \b in) {$L$};

\draw[black, fill = black] (105: .5*\rrr in) circle (.3ex);
\node[shift = {(2ex, -1ex)}] at (105: .5*\rrr in) {$P$};

\draw[black, fill = black] (135: .866*\rrr in) circle (.3ex);
\node[shift = {(2.5ex, 0)}] at (135: .866*\rrr in) {$M$};

\begin{scope}[shift = {(120: \crrr in)}]
\begin{scope}[rotate = 45]

\draw[black, fill = black] (\drrr in, 0) circle (.3ex);
\node[shift = {(-2ex, 0)}] at (\drrr in, 0) {$K$};
\end{scope}
\end{scope}

\begin{scope}[shift = {(120: \crrr in)}]

\draw[black, fill = black] (135: 2*\length in) circle (.3ex);
\node[shift = {(-2ex, 0)}] at (135: 2*\length in) {$D$};

\begin{scope}[shift = {(45:\drrr in)}]

\draw[black, fill = black] (135: 2.97*\length in) circle (.3ex);
\node[shift = {(-2ex, 0)}] at (135: 2.97*\length in) {$E$};

\begin{scope}[shift = {(45:.195*\drrr in)}]

\draw[black, fill = black] (135: .488*\rrr in) circle (.3ex);
\node[shift = {(2ex, .5ex)}] at (135: .488*\rrr in) {$F$};

\end{scope}

\end{scope}

\end{scope}

\end{scope}

\end{scope}

}

\newcommand{\pipicg}{

 \pgfmathsetmacro{\l}{\linewidth/28.440}%
 \pgfmathsetmacro{\h}{\l/2.65}%

 \pgfmathsetmacro{\ww}{\www/2}%
 \pgfmathsetmacro{\hh}{\hhh/2}%
 \pgfmathsetmacro{\rrr}{\dwscale}%

\pgfmathsetmacro{\newr}{1/cos(55)}
\pgfmathsetmacro{\newrr}{1/cos(32.5)}
\pgfmathsetmacro{\hyp}{\rrr*\newr}
\pgfmathsetmacro{\nhyp}{\rrr*\newrr}
\pgfmathsetmacro{\a}{1/cos(30)}
\pgfmathsetmacro{\b}{\rrr*\a}
\pgfmathsetmacro{\length}{\b - \rrr}
\pgfmathsetmacro{\c}{1/cos(15)}
\pgfmathsetmacro{\crrr}{\rrr*\c}
\pgfmathsetmacro{\d}{sin(15)}
\pgfmathsetmacro{\drrr}{\rrr*\d}

\begin{scope}[shift = {(.17in, -1.29in)}]

\begin{scope}[shift = {(\ww in,\hh in)}]

\begin{scope}[rotate = 10]

\draw[line width = .75pt, \dwcblue] (20: \rrr in) arc (20:65:\rrr in);

\draw[dashed, line width = .75pt, \dwcblue] (65: \rrr in) arc (65:80:\rrr in);

\draw[line width = .75pt, \dwcblue] (80: \rrr in) arc (80:110:\rrr in);

\draw[dashed, line width = .75pt, \dwcblue] (110: \rrr in) arc (110:125:\rrr in);

\draw[line width = .75pt, \dwcblue] (125: \rrr in) arc (125:170:\rrr in);

\draw[shorten <=-30, shorten >=-20, black] (40: \rrr in)--(95: \hyp in);      
\draw[shorten <=-30, shorten >=-20, black] (150: \rrr in)--(95: \hyp in);
\draw[black] (85: \rrr in)--(117.5: \nhyp in); 
\draw[black] (40: \rrr in) -- (150: \rrr in);
\draw[black] (85: \rrr in) -- (150: \rrr in);
     
\draw[black, fill = black] (0, 0) circle (.4ex);  
\node[shift = {(2ex, 0)}] at (0,0) {$O$};

\draw[black, fill = black] (40: \rrr in) circle (.3ex);  
\node[shift = {(2ex,2.5ex)}, fill = white, inner sep = 1.5pt] at (40: \rrr in) {$P_{n+1}$};

\draw[black, fill = black] (60: \rrr in) circle (.3ex);  
\node[shift = {(0, -2.5ex)}, fill = white, inner sep = 1.5pt] at (60: \rrr in) {$P_n$};

\draw[black, fill = black] (85: \rrr in) circle (.3ex);  
\node[shift = {(0, -2.5ex)}, fill = white, inner sep = 1.5pt] at (80: \rrr in) {$P_{m+1}$};

\draw[black, fill = black] (105: \rrr in) circle (.3ex);  
\node[shift = {(0,-2.5ex)}, fill = white, inner sep = 1.5pt] at (100: \rrr in) {$P_m$};

\draw[black, fill = black] (130: \rrr in) circle (.3ex);  
\node[shift = {(0, -2.5ex)}, fill = white, inner sep = 1.5pt] at (130: \rrr in) {$P_2$};

\draw[black, fill = black] (150: \rrr in) circle (.3ex);  
\node[shift = {(-2.5ex, 0)}, fill = white, inner sep = 1.5pt] at (150: \rrr in) {$P_1$};

\draw[black, fill = black] (117.5: \nhyp in) circle (.3ex);  
\node[shift = {(-2.5ex, 0)}] at (117.5: \nhyp in) {$A$};

\draw[black, fill = black] (95: \hyp in) circle (.3ex);  
\node[shift = {(-2.5ex, 0)},  inner sep = 14pt] at (95: \hyp in) {$B$};

\end{scope}

\end{scope}

\end{scope}

}

\newcommand{\pipich}{

 \pgfmathsetmacro{\l}{\linewidth/28.440}%
 \pgfmathsetmacro{\h}{\l/2.65}%

 \pgfmathsetmacro{\ww}{\www/2}%
 \pgfmathsetmacro{\hh}{\hhh/2}%
 \pgfmathsetmacro{\rrr}{\dwscale}%

\pgfmathsetmacro{\newr}{cos(65)}
\pgfmathsetmacro{\htt}{\rrr*\newr}

\pgfmathsetmacro{\newrr}{cos(35)}
\pgfmathsetmacro{\httt}{\rrr*\newrr}

\begin{scope}[shift = {(\ww in,\hh in)}]

\begin{scope}[shift = {(-.28in, -1.105in)}]

\begin{scope}[rotate = -21]

\draw[line width = .75pt, \dwcblue] (20: \rrr in) arc (20:70:\rrr in);

\draw[dashed, line width = .75pt, \dwcblue] (70: \rrr in) arc (75:87:\rrr in);

\draw[dashed, line width = .75pt, \dwcblue] (108: \rrr in) arc (108:120:\rrr in);

\draw[line width = .75pt, \dwcblue] (120: \rrr in) arc (120:170:\rrr in);

\draw[black] (30: \rrr in) -- (160: \rrr in);
\draw[black] (30: \rrr in) -- (60: \rrr in);
\draw[black] (160: \rrr in) -- (130: \rrr in);
\draw[black] (60: \rrr in) -- (130: \rrr in);

\draw[\dwccyan] (0,0) -- (30: \rrr in);

\draw[\dwccyan] (0,0) -- (60: \rrr in);

\draw[\dwcviolet] (0,0) -- (95: \httt in);

\draw[\dwccyan] (0,0) -- (130: \rrr in);

\draw[\dwccyan] (0,0) -- (160: \rrr in);

\draw[black, fill = black] (0, 0) circle (.4ex);  
\node[shift = {(-2ex, 0)}] at (0, 0) {$O$};

\draw[black, fill = black] (30: \rrr in) circle (.3ex);  
\node[shift = {(0: 2.5ex)}] at (30: \rrr in) {$B$};

\draw[black, fill = black] (60: \rrr in) circle (.3ex);  
\node[shift = {(30: 2.5ex)}] at (60: \rrr in) {$D$};

\draw[black, fill = black] (130: \rrr in) circle (.3ex);  
\node[shift = {(100: 2.5ex)}] at (130: \rrr in) {$C$};

\draw[black, fill = black] (160: \rrr in) circle (.3ex);  
\node[inner sep = 0pt, shift = {(130: 2.5ex)},  inner sep = 0pt] at (160: \rrr in) {$A$};

\draw[black, fill = black] (95: \htt in) circle (.3ex);  
\node[shift = {(25: 2.5ex)}] at (95: \htt in) {$E$};

\draw[black, fill = black] (95: \httt in) circle (.3ex);  
\node[shift = {(25: 2.5ex)}] at (95: \httt in) {$F$};

\end{scope}

\end{scope}

\end{scope}

}

\newcommand{\pipichh}{

 \pgfmathsetmacro{\l}{\linewidth/28.440}%
 \pgfmathsetmacro{\h}{\l/2.65}%

 \pgfmathsetmacro{\ww}{\www/2}%
 \pgfmathsetmacro{\hh}{\hhh/2}%
 \pgfmathsetmacro{\rrr}{\dwscale}%

\pgfmathsetmacro{\newr}{cos(65)}
\pgfmathsetmacro{\htt}{\rrr*\newr}

\pgfmathsetmacro{\newrr}{cos(35)}
\pgfmathsetmacro{\httt}{\rrr*\newrr}

\begin{scope}[shift = {(\ww in,\hh in)}]

\begin{scope}[shift = {(-.28in, -1.105in)}]

\begin{scope}[rotate = -21]

\draw[line width = .75pt, \dwcblue] (20: \rrr in) arc (20:170:\rrr in);

\draw[] (30: \rrr in) -- (160: \rrr in);
\draw[] (30: \rrr in) -- (95: \rrr in);
\draw[] (160: \rrr in) -- (95: \rrr in);

\draw[\dwccyan] (0,0) -- (30: \rrr in);

\draw[\dwccyan] (0,0) -- (95: \rrr in);

\draw[\dwccyan] (0,0) -- (160: \rrr in);

\draw[black, fill = black] (0, 0) circle (.4ex);  
\node[shift = {(-2ex, 0)}] at (0, 0) {$O$};

\draw[black, fill = black] (30: \rrr in) circle (.3ex);  
\node[shift = {(0: 2.5ex)}] at (30: \rrr in) {$B$};

\draw[black, fill = black] (160: \rrr in) circle (.3ex);  
\node[inner sep = 0pt, shift = {(130: 2.5ex)},  inner sep = 0pt] at (160: \rrr in) {$A$};

\draw[black, fill = black] (95: \htt in) circle (.3ex);  
\node[shift = {(25: 2.5ex)}] at (95: \htt in) {$E$};

\draw[black, fill = black] (95: \rrr in) circle (.3ex);  
\node[shift = {(65: 2.5ex)}] at (95: \rrr in) {$C$};

\end{scope}

\end{scope}

\end{scope}

}

\newcommand{\pipici}{

 \pgfmathsetmacro{\l}{\linewidth/28.440}%
 \pgfmathsetmacro{\h}{\l/2.65}%

 \pgfmathsetmacro{\ww}{\www/2}%
 \pgfmathsetmacro{\hh}{\hhh/2}%
 \pgfmathsetmacro{\rrr}{\dwscale}%

\pgfmathsetmacro{\newr}{1/cos(55)}
\pgfmathsetmacro{\newrr}{1/cos(44)}
\pgfmathsetmacro{\news}{1/cos(33)}
\pgfmathsetmacro{\newt}{1/cos(22)}
\pgfmathsetmacro{\newu}{1/cos(11)}

\pgfmathsetmacro{\newuu}{tan(44)}
\pgfmathsetmacro{\newtt}{tan(33)}
\pgfmathsetmacro{\newss}{tan(22)}
\pgfmathsetmacro{\newrrr}{tan(11)}

\pgfmathsetmacro{\hyp}{\rrr*\newr}
\pgfmathsetmacro{\hypp}{\rrr*\newrr}
\pgfmathsetmacro{\hyppp}{\rrr*\newrrr}
\pgfmathsetmacro{\hyps}{\rrr*\news}
\pgfmathsetmacro{\hypt}{\rrr*\newt}
\pgfmathsetmacro{\hypss}{\rrr*\newss}
\pgfmathsetmacro{\hyptt}{\rrr*\newtt}
\pgfmathsetmacro{\hypu}{\rrr*\newu}
\pgfmathsetmacro{\hypuu}{\rrr*\newuu}
\begin{scope}[shift = {(.184in, -1.29in)}]

\begin{scope}[shift = {(\ww in,\hh in)}]

\begin{scope}[rotate = 20]

\draw[line width = .75pt, \dwcblue] (20: \rrr in) arc (20:70:\rrr in);

\draw[dashed, line width = .75pt, \dwcred] (70: \rrr in) arc (70:100:\rrr in);

\draw[line width = .75pt, \dwcblue] (100: \rrr in) arc (100:170:\rrr in);

\draw[shorten <=-30, shorten >=-20, black] (40: \rrr in)--(95: \hyp in);      
\draw[shorten <=-30, shorten >=-20, black] (150: \rrr in)--(95: \hyp in);

\draw[\dwcviolet, shorten <=-\hyppp in] (62: \rrr in) -- (106:\hypp in);
\draw[\dwcorange, shorten <=-\hypss in] (84: \rrr in) -- (117:\hyps in); 
\draw[\dwcpurple, shorten <=-\hyptt in]  (106: \rrr in) -- (128: \hypt in);
\draw[\dwcgreen!50!black, shorten <=-\hypuu in] (128: \rrr in) -- (139: \hypu in);

\draw[\dwcteal] (0,0) -- (95: \hyp in);

\draw[black, fill = black] (51: \hypu in) circle (.2ex);  
\node[shift = {(4ex, 1ex)}, fill = white, inner sep = 1.5pt] at (51: \hypu in) {$P_{n,n-1}$};

\draw[black, fill = black] (62: \hypt in) circle (.2ex);  
\node[shift = {(3ex, 1ex)}, fill = white, inner sep = 1.5pt, rotate = -30] at (62: \hypt in) {$\dots$};

\draw[black, fill = black] (73: \hypu in) circle (.2ex);  

\draw[black, fill = black] (73: \hyps in) circle (.2ex);  
\node[shift = {(4ex, 1ex)}, fill = white, inner sep = 1.5pt] at (73: \hyps in) {$P_{n,3}$};

\draw[black, fill = black] (84: \hypt in) circle (.2ex);  
\node[font = \scriptsize, shift = {(1ex,2ex)}, fill = white, inner sep = 1pt] at (84: \hypt in) {$P_{n-1,3}$};

\draw[black, fill = black] (84: \hypp in) circle (.2ex);  
\node[shift = {(4ex, 1ex)}, fill = white, inner sep = 1.5pt] at (84: \hypp in) {$P_{n,2}$};

\draw[black, fill = black] (95: \hypu in) circle (.2ex);  

\draw[black, fill = black] (95: \hyps in) circle (.2ex);  
\node[font = \scriptsize, shift = {(-1.5ex,2ex)}, fill = white, inner sep = 1pt] at (95: \hyps in) {$P_{2, n-1}$};

\draw[black, fill = black] (106: \hypt in) circle (.2ex);  
\node[font = \scriptsize, shift = {(-1.5ex,2ex)}, fill = white, inner sep = 1pt] at (106: \hypt in) {$P_{2, n-2}$};

\draw[black, fill = black] (106: \hypp in) circle (.2ex);  
\node[shift = {(-4ex,1ex)}, fill = white, inner sep = 1.5pt] at (106: \hypp in) {$P_{1,n-1}$};

\draw[black, fill = black] (117: \hypu in) circle (.2ex);  

\draw[black, fill = black] (117: \hyps in) circle (.2ex);  
\node[shift = {(-4ex,1ex)}, fill = white, inner sep = 1.5pt] at (117: \hyps in) {$P_{1, n-2}$};

\draw[black, fill = black] (128: \hypt in) circle (.2ex);  
\node[shift = {(-3ex,1ex)}, fill = white, inner sep = 1.5pt, rotate = 80] at (128: \hypt in) {$\dots$};

\draw[black, fill = black] (139: \hypu in) circle (.2ex);  
\node[shift = {(-4ex,1ex)}, fill = white, inner sep = 1.5pt] at (139: \hypu in) {$P_{1,2}$};

\draw[black, fill = black] (0, 0) circle (.4ex);  
\node[shift = {(2ex, 0)}] at (0,0) {$O$};

\draw[black, fill = black] (40: \rrr in) circle (.3ex);  
\node[shift = {(4ex, 1ex)}, fill = white, inner sep = 1.5pt] at (40: \rrr in) {$P_{n}$};

\draw[black, fill = black] (62: \rrr in) circle (.3ex);  
\node[shift = {(0, -2.5ex)}, fill = white, inner sep = 1.5pt] at (62: \rrr in) {$P_{n-1}$};

\node[shift = {(0, -2.5ex)}, fill = white, inner sep = 1.5pt, rotate = 15] at (84: \rrr in) {$\dots$};

\draw[black, fill = black] (106: \rrr in) circle (.3ex);  
\node[shift = {(0,-2.5ex)}, fill = white, inner sep = 1.5pt] at (106: \rrr in) {$P_3$};

\draw[black, fill = black] (128: \rrr in) circle (.3ex);  
\node[shift = {(0, -2.5ex)}, fill = white, inner sep = 1.5pt] at (128: \rrr in) {$P_2$};

\draw[black, fill = black] (150: \rrr in) circle (.3ex);  
\node[shift = {(-2.5ex, 0)}, fill = white, inner sep = 1.5pt] at (150: \rrr in) {$P_1$};

\draw[black, fill = black] (95: \hyp in) circle (.3ex);  
\node[shift = {(-4ex, 0ex)},  inner sep = 13pt] at (95: \hyp in) {$P_{1,n}$};
\node[shift = {(4ex, 1ex)}] at (95: \hyp in) {$P_{n,1}$};

\end{scope}

\end{scope}

\end{scope}

}

\newcommand{\pipicj}{

 \pgfmathsetmacro{\l}{\linewidth/28.440}%
 \pgfmathsetmacro{\h}{\l/2.65}%

 \pgfmathsetmacro{\ww}{\www/2}%
 \pgfmathsetmacro{\hh}{\hhh/2}%
 \pgfmathsetmacro{\rrr}{\dwscale}%

\pgfmathsetmacro{\newr}{1/cos(55)}
\pgfmathsetmacro{\news}{1/cos(25)}
\pgfmathsetmacro{\hyps}{\rrr*\news}

\pgfmathsetmacro{\aaa}{cos(55)}
\pgfmathsetmacro{\aaaa}{cos(25)}

\pgfmathsetmacro{\bbb}{\aaa*\rrr}
\pgfmathsetmacro{\bbbb}{\aaaa*\rrr}

\pgfmathsetmacro{\hyp}{\rrr*\newr}

\pgfmathsetmacro{\newrr}{1/cos(40)}
\pgfmathsetmacro{\newrrr}{tan(15)}
\pgfmathsetmacro{\hypp}{\rrr*\newrr}
\pgfmathsetmacro{\hyppp}{\rrr*\newrrr}
\pgfmathsetmacro{\LC}{\hyppp + .5}

\pgfmathsetmacro{\newu}{1/cos(15)}
\pgfmathsetmacro{\newuu}{tan(40)}
\pgfmathsetmacro{\hypu}{\rrr*\newu}
\pgfmathsetmacro{\hypuu}{\rrr*\newuu}
\pgfmathsetmacro{\LB}{\hypuu + .5}

\begin{scope}[shift = {(-.76in, -.597in)}]

\begin{scope}[shift = {(\ww in,\hh in)}]

\begin{scope}[rotate = -50]

\draw[line width = .75pt, \dwcblue] (20: \rrr in) arc (20:170:\rrr in);

\draw[shorten <=-30, shorten >=-20, black] (40: \rrr in)--(95: \hyp in);      
\draw[shorten <=-30, shorten >=-20, black] (150: \rrr in)--(95: \hyp in);

\draw[\dwcviolet, shorten <=-\LC in] (70: \rrr in) -- (110:\hypp in); 

\draw[\dwcgreen!50!black, shorten <=-\LB in] (120: \rrr in) -- (135: \hypu in); 

\draw[\dwcred] (150:\rrr in)--(120:\rrr in);
\draw[\dwcred] (70:\rrr in)--(120:\rrr in);
\draw[\dwcred] (70:\rrr in)--(40:\rrr in);
\draw[\dwcred] (40:\rrr in)--(150:\rrr in);

\draw[\dwcteal] (0,0)--(150:\rrr in);
\draw[\dwcteal] (0,0)--(120:\rrr in);
\draw[\dwcteal] (0,0)--(70:\rrr in);
\draw[\dwcteal] (0,0)--(40:\rrr in);

\draw[\dwcteal, shorten >= -.4in] (0,0) -- (95: \hyp in);

\draw[black, fill = black] (55: \hypu in) circle (.2ex);  
\node[shift = {(5: 3ex)}, fill = white, inner sep = 1pt] at (55: \hypu in) {$P_{CD}$};
\node[shift = {(-70: 5ex)}, fill = white, inner sep = 1pt] at (55: \hypu in) {$L_C$};

\draw[black, fill = black] (80: \hypp in) circle (.2ex);  
\node[font = \footnotesize, shift = {(200: 3ex)}, fill = white, inner sep = 1pt] at (80: \hypp in) {$P_{BD}$};
\node[shift = {(-33: 3.85ex)}, fill = white, inner sep = 1pt] at (80: \hypp in) {$L_B$};

\draw[black, fill = black] (95: \hyps in) circle (.2ex);  
\node[font = \footnotesize, shift = {(110: 3ex)}, fill = white, inner sep = 1pt] at (95: \hyps in) {$P_{BC}$};

\draw[black, fill = black] (110: \hypp in) circle (.2ex);  
\node[shift = {(140:3ex)}, fill = white, inner sep = 1pt] at (110: \hypp in) {$P_{AC}$};

\draw[black, fill = black] (135: \hypu in) circle (.2ex);  
\node[shift = {(140: 3ex)}, fill = white, inner sep = 1pt] at (135: \hypu in) {$P_{AB}$};

\draw[black, fill = black] (0, 0) circle (.4ex);  
\node[shift = {(-2ex, 0)}, fill = white, inner sep = 1pt] at (0,0) {$O$};

\draw[black, fill = black] (40: \rrr in) circle (.3ex);  
\node[shift = {(0: -2ex)}, fill = white, inner sep = 1pt] at (40: \rrr in) {$D$};
\node[shift = {(-80: 6ex)}, fill = white, inner sep = 1pt] at (40: \rrr in) {$L_D$};

\draw[black, fill = black] (70: \rrr in) circle (.3ex);  
\node[shift = {(10: -2ex)}, fill = white, inner sep = 1pt] at (70: \rrr in) {$C$};

\draw[black, fill = black] (120: \rrr in) circle (.3ex);  
\node[shift = {(70: -2ex)}, fill = white, inner sep = 1pt] at (120: \rrr in) {$B$};

\draw[black, fill = black] (150: \rrr in) circle (.3ex);  
\node[shift = {(100: -2ex)}, fill = white, inner sep = 1pt] at (150: \rrr in) {$A$};

\draw[black, fill = black] (95: \hyp in) circle (.3ex);  

\node at (95: \bbb in) [shift = {(270: 1.41ex)}, rotate = -45, draw,thin,minimum width=2ex,minimum height=2ex] {};

\node at (95: \bbbb in) [shift = {(270: 1.41ex)}, rotate = -45, draw,thin,minimum width=2ex,minimum height=2ex] {};

\draw[black, fill = black] (95: \bbb in) circle (.3ex);
\node[shift = {(185: 2.5ex)}] at (95: \bbb in) {$E$};

\draw[black, fill = black] (95: \bbbb in) circle (.3ex);
\node[shift = {(185: 2.5ex)}] at (95: \bbbb in) {$F$};

\node[shift = {(50: 3ex)}, fill = white, inner sep = 1pt] at (95: \hyp in) {$L$};
\node[shift = {(-2.7: 4.6ex)}, fill = white, inner sep = 0pt] at (95: \hyp in) {$L_A$};
\node[shift = {(140: 3ex)},  inner sep = 5pt] at (95: \hyp in) {$P_{AD}$};

\end{scope}

\end{scope}

\end{scope}

}

\newcommand{\pipicnewk}{

 \pgfmathsetmacro{\l}{\linewidth/28.440}%
 \pgfmathsetmacro{\h}{\l/2.65}%

 \pgfmathsetmacro{\ww}{\www/2}%
 \pgfmathsetmacro{\hh}{\hhh/2}%
 \pgfmathsetmacro{\rrr}{\dwscale}%

\newcommand{\dwangk}{30}

\coordinate(O) at (0,0);

\coordinate (Pa) at (170: \rrr in);
\path (Pa);
\pgfgetlastxy{\Pax}{\Pay};
\coordinate (Pb) at (150: \rrr in);
\path (Pb);
\pgfgetlastxy{\Pbx}{\Pby};
\coordinate (Pc) at (130: \rrr in);
\path (Pc);
\pgfgetlastxy{\Pcx}{\Pcy};
\coordinate (Pd) at (110: \rrr in);
\path (Pd);
\pgfgetlastxy{\Pdx}{\Pdy};
\coordinate (Pe) at (90: \rrr in);
\path (Pe);
\pgfgetlastxy{\Pex}{\Pey};
\coordinate (Pf) at (70: \rrr in);
\path (Pf);
\pgfgetlastxy{\Pfx}{\Pfy};
\coordinate (Pg) at (50: \rrr in);
\path (Pg);
\pgfgetlastxy{\Pgx}{\Pgy};
\coordinate (Ph) at (30: \rrr in);
\path (Ph);
\pgfgetlastxy{\Phx}{\Phy};
\coordinate (Pi) at (10: \rrr in);
\path (Pi);
\pgfgetlastxy{\Pix}{\Piy};

\coordinate (ab) at (\Pbx, \Pay);
\coordinate (ac) at (\Pcx, \Pay);
\coordinate (ad) at (\Pdx, \Pay);
\coordinate (ae) at (\Pex, \Pay);
\coordinate (af) at (\Pfx, \Pay);
\coordinate (ag) at (\Pgx, \Pay);
\coordinate (ah) at (\Phx, \Pay);
\coordinate (ai) at (\Pix, \Pay);

\coordinate (X2) at (\Pcx, \Pby);
\coordinate (X3) at (\Pdx, \Pcy);
\coordinate (X4) at (\Pex, \Pdy);
\coordinate (X5) at (\Pfx, \Pgy);
\coordinate (X6) at (\Pgx, \Phy);
\coordinate (X7) at (\Pgx, \Phy);

\coordinate (rad) at (90:\rrr in);

\draw[black] ([rotate around={\dwangk:(0,0)}]O) -- ([rotate around={\dwangk:(0,0)}]rad); 
\draw[fill = black, black] ([rotate around={\dwangk:(0,0)}]O) circle (.4ex); 
\node[shift = {(1.5ex, -1.5ex)}] at ([rotate around={\dwangk:(0,0)}]O) {$O$};

\draw[line width = .5pt, \dwcred, rotate around={\dwangk:(0,0)}] (10: \rrr in) arc (10:170:\rrr in);

\draw[blue, thick] ([rotate around={\dwangk:(0,0)}]Pa) -- ([rotate around={\dwangk:(0,0)}]Pb) -- ([rotate around={\dwangk:(0,0)}]Pc) -- ([rotate around={\dwangk:(0,0)}]Pd) -- ([rotate around={\dwangk:(0,0)}]Pe) -- ([rotate around={\dwangk:(0,0)}]Pf) -- ([rotate around={\dwangk:(0,0)}]Pg) -- ([rotate around={\dwangk:(0,0)}]Ph) -- ([rotate around={\dwangk:(0,0)}]Pi);

\draw[] ([rotate around={\dwangk:(0,0)}]Pb) -- ([rotate around={\dwangk:(0,0)}]X2);
\node[shift = {(.45, -.1)}, teal] at ([rotate around={\dwangk:(0,0)}]X2) {$X_2$};
\draw[] ([rotate around={\dwangk:(0,0)}]Pc) -- ([rotate around={\dwangk:(0,0)}]X3);
\node[shift = {(.45, -.1)}, teal] at ([rotate around={\dwangk:(0,0)}]X3) {$X_3$};
\draw[] ([rotate around={\dwangk:(0,0)}]Pd) -- ([rotate around={\dwangk:(0,0)}]X4);
\node[shift = {(.45, -.1)}, teal] at ([rotate around={\dwangk:(0,0)}]X4) {$X_4$};
\draw[] ([rotate around={\dwangk:(0,0)}]Pg) -- ([rotate around={\dwangk:(0,0)}]X5);
\node[shift = {(.45, -.1)}, teal] at ([rotate around={\dwangk:(0,0)}]X5) {$X_5$};
\draw[] ([rotate around={\dwangk:(0,0)}]Ph) -- ([rotate around={\dwangk:(0,0)}]X6);
\node[shift = {(.45, -.1)}, teal] at ([rotate around={\dwangk:(0,0)}]X6) {$X_6$};

\node[shift = {(170: 2ex)}] at ([rotate around={\dwangk:(0,0)}]Pa) {$P_1$};
\node[shift = {(150: 2ex)}] at ([rotate around={\dwangk:(0,0)}]Pb) {$P_2$};
\node[shift = {(130: 2ex)}] at ([rotate around={\dwangk:(0,0)}]Pc) {$P_3$};
\node[shift = {(110: 2ex)}] at ([rotate around={\dwangk:(0,0)}]Pd) {$P_4$};
\node[shift = {(90: 2ex)}] at ([rotate around={\dwangk:(0,0)}]Pe) {$P_5$};
\node[shift = {(70: 2ex)}] at ([rotate around={\dwangk:(0,0)}]Pf) {$P_6$};
\node[shift = {(50: 2ex)}] at ([rotate around={\dwangk:(0,0)}]Pg) {$P_7$};
\node[shift = {(30: 2ex)}] at ([rotate around={\dwangk:(0,0)}]Ph) {$P_8$};
\node[shift = {(10: 2ex)}] at ([rotate around={\dwangk:(0,0)}]Pi) {$P_9$};

\node[blue, shift = {(0, -2ex)}] at ([rotate around={\dwangk:(0,0)}]ab) {$q_2$};
\node[blue, shift = {(0, -2ex)}] at ([rotate around={\dwangk:(0,0)}]ac) {$q_3$};
\node[blue, shift = {(0, -2ex)}] at ([rotate around={\dwangk:(0,0)}]ad) {$q_4$};
\node[blue, shift = {(0, -2ex)}] at ([rotate around={\dwangk:(0,0)}]ae) {$q_5$};
\node[blue, shift = {(0, -2ex)}] at ([rotate around={\dwangk:(0,0)}]af) {$q_6$};
\node[blue, shift = {(0, -2ex)}] at ([rotate around={\dwangk:(0,0)}]ag) {$q_7$};
\node[blue, shift = {(0, -2ex)}] at ([rotate around={\dwangk:(0,0)}]ah) {$q_8$};

\draw[] ([rotate around={\dwangk:(0,0)}]Pa) -- ([rotate around={\dwangk:(0,0)}]Pi);

\draw[] ([rotate around={\dwangk:(0,0)}]Pb) -- ([rotate around={\dwangk:(0,0)}]ab);
\draw[] ([rotate around={\dwangk:(0,0)}]Pc) -- ([rotate around={\dwangk:(0,0)}]ac);
\draw[] ([rotate around={\dwangk:(0,0)}]Pd) -- ([rotate around={\dwangk:(0,0)}]ad);
\draw[] ([rotate around={\dwangk:(0,0)}]Pe) -- ([rotate around={\dwangk:(0,0)}]ae);
\draw[] ([rotate around={\dwangk:(0,0)}]Pf) -- ([rotate around={\dwangk:(0,0)}]af);
\draw[] ([rotate around={\dwangk:(0,0)}]Pg) -- ([rotate around={\dwangk:(0,0)}]ag);
\draw[] ([rotate around={\dwangk:(0,0)}]Ph) -- ([rotate around={\dwangk:(0,0)}]ah);
\draw[] ([rotate around={\dwangk:(0,0)}]Pi) -- ([rotate around={\dwangk:(0,0)}]ai);

\begin{scope}[rotate around={\dwangk:(0,0)}]
\draw[shift = {([rotate around={\dwangk:(0,0)}]ab)}] (0,0) rectangle (1.25ex, 1.25ex);
\draw[shift = {([rotate around={\dwangk:(0,0)}]ac)}] (0,0) rectangle (1.25ex, 1.25ex);
\draw[shift = {([rotate around={\dwangk:(0,0)}]ad)}] (0,0) rectangle (1.25ex, 1.25ex);
\draw[shift = {([rotate around={\dwangk:(0,0)}]ae)}] (0,0) rectangle (1.25ex, 1.25ex);
\draw[shift = {([rotate around={\dwangk:(0,0)}]af)}] (0,0) rectangle (1.25ex, 1.25ex);
\draw[shift = {([rotate around={\dwangk:(0,0)}]ag)}] (0,0) rectangle (1.25ex, 1.25ex);
\draw[shift = {([rotate around={\dwangk:(0,0)}]ah)}] (0,0) rectangle (1.25ex, 1.25ex);
\end{scope}

\draw[fill = black] ([rotate around={\dwangk:(0,0)}]Pa) circle (.25ex);
\draw[fill = black] ([rotate around={\dwangk:(0,0)}]Pb) circle (.25ex);
\draw[fill = black] ([rotate around={\dwangk:(0,0)}]Pc) circle (.25ex);
\draw[fill = black] ([rotate around={\dwangk:(0,0)}]Pd) circle (.25ex);
\draw[fill = black] ([rotate around={\dwangk:(0,0)}]Pe) circle (.25ex);
\draw[fill = black] ([rotate around={\dwangk:(0,0)}]Pf) circle (.25ex);
\draw[fill = black] ([rotate around={\dwangk:(0,0)}]Pg) circle (.25ex);
\draw[fill = black] ([rotate around={\dwangk:(0,0)}]Ph) circle (.25ex);
\draw[fill = black] ([rotate around={\dwangk:(0,0)}]Pi) circle (.25ex);

\draw[blue, fill = blue] ([rotate around={\dwangk:(0,0)}]ab) circle (.25ex);
\draw[blue, fill = blue] ([rotate around={\dwangk:(0,0)}]ac) circle (.25ex);
\draw[blue, fill = blue] ([rotate around={\dwangk:(0,0)}]ad) circle (.25ex);
\draw[blue, fill = blue] ([rotate around={\dwangk:(0,0)}]ae) circle (.25ex);
\draw[blue, fill = blue] ([rotate around={\dwangk:(0,0)}]af) circle (.25ex);
\draw[blue, fill = blue] ([rotate around={\dwangk:(0,0)}]ag) circle (.25ex);
\draw[blue, fill = blue] ([rotate around={\dwangk:(0,0)}]ah) circle (.25ex);
\draw[] ([rotate around={\dwangk:(0,0)}]ai) circle (.25ex);

\draw[teal, fill = teal] ([rotate around={\dwangk:(0,0)}]X2) circle (.25ex);
\draw[teal, fill = teal] ([rotate around={\dwangk:(0,0)}]X3) circle (.25ex);
\draw[teal, fill = teal] ([rotate around={\dwangk:(0,0)}]X4) circle (.25ex);
\draw[teal, fill = teal] ([rotate around={\dwangk:(0,0)}]X5) circle (.25ex);
\draw[teal, fill = teal] ([rotate around={\dwangk:(0,0)}]X6) circle (.25ex);

}

\newcommand{\pipicnewl}{

 \pgfmathsetmacro{\l}{\linewidth/28.440}%
 \pgfmathsetmacro{\h}{\l/2.65}%

 \pgfmathsetmacro{\ww}{\www/2}%
 \pgfmathsetmacro{\hh}{\hhh/2}%
 \pgfmathsetmacro{\rrr}{\dwscale}%

\newcommand{\dwangk}{30}

\coordinate(O) at (0,0);

\coordinate (Pa) at (171: \rrr in);
\path (Pa);
\pgfgetlastxy{\Pax}{\Pay};
\coordinate (Pb) at (153: \rrr in);
\path (Pb);
\pgfgetlastxy{\Pbx}{\Pby};
\coordinate (Pc) at (135: \rrr in);
\path (Pc);
\pgfgetlastxy{\Pcx}{\Pcy};
\coordinate (Pd) at (117: \rrr in);
\path (Pd);
\pgfgetlastxy{\Pdx}{\Pdy};
\coordinate (Pe) at (99: \rrr in);
\path (Pe);
\pgfgetlastxy{\Pex}{\Pey};
\coordinate (Pf) at (81: \rrr in);
\path (Pf);
\pgfgetlastxy{\Pfx}{\Pfy};
\coordinate (Pg) at (63: \rrr in);
\path (Pg);
\pgfgetlastxy{\Pgx}{\Pgy};
\coordinate (Ph) at (45: \rrr in);
\path (Ph);
\pgfgetlastxy{\Phx}{\Phy};
\coordinate (Pi) at (27: \rrr in);
\path (Pi);
\pgfgetlastxy{\Pix}{\Piy};
\coordinate (Pj) at (9: \rrr in);
\path (Pj);
\pgfgetlastxy{\Pjx}{\Pjy};

\coordinate (ab) at (\Pbx, \Pay);
\coordinate (ac) at (\Pcx, \Pay);
\coordinate (ad) at (\Pdx, \Pay);
\coordinate (ae) at (\Pex, \Pay);
\coordinate (af) at (\Pfx, \Pay);
\coordinate (ag) at (\Pgx, \Pay);
\coordinate (ah) at (\Phx, \Pay);
\coordinate (ai) at (\Pix, \Pay);
\coordinate (aj) at (\Pjx, \Pay);

\coordinate (X2) at (\Pcx, \Pby);
\coordinate (X3) at (\Pdx, \Pcy);
\coordinate (X4) at (\Pex, \Pdy);
\coordinate (X5) at (\Pfx, \Pgy);
\coordinate (X6) at (\Pgx, \Phy);
\coordinate (X7) at (\Phx, \Piy);
\coordinate (X8) at (\Pix, \Pjy);

\coordinate (rad) at (90:\rrr in);

\draw[black] ([rotate around={\dwangk:(0,0)}]O) -- ([rotate around={\dwangk:(0,0)}]rad); 
\draw[fill = black, black] ([rotate around={\dwangk:(0,0)}]O) circle (.4ex); 
\node[shift = {(1.5ex, -1.5ex)}] at ([rotate around={\dwangk:(0,0)}]O) {$O$};

\draw[line width = .5pt, \dwcred, rotate around={\dwangk:(0,0)}] (9: \rrr in) arc (9:171:\rrr in);

\draw[blue, thick] ([rotate around={\dwangk:(0,0)}]Pa) -- ([rotate around={\dwangk:(0,0)}]Pb) -- ([rotate around={\dwangk:(0,0)}]Pc) -- ([rotate around={\dwangk:(0,0)}]Pd) -- ([rotate around={\dwangk:(0,0)}]Pe) -- ([rotate around={\dwangk:(0,0)}]Pf) -- ([rotate around={\dwangk:(0,0)}]Pg) -- ([rotate around={\dwangk:(0,0)}]Ph) -- ([rotate around={\dwangk:(0,0)}]Pi)-- ([rotate around={\dwangk:(0,0)}]Pj);

\draw[] ([rotate around={\dwangk:(0,0)}]Pb) -- ([rotate around={\dwangk:(0,0)}]X2);
\node[shift = {(.45, -.1)}, teal] at ([rotate around={\dwangk:(0,0)}]X2) {$X_2$};
\draw[] ([rotate around={\dwangk:(0,0)}]Pc) -- ([rotate around={\dwangk:(0,0)}]X3);
\node[shift = {(.45, -.1)}, teal] at ([rotate around={\dwangk:(0,0)}]X3) {$X_3$};
\draw[] ([rotate around={\dwangk:(0,0)}]Pd) -- ([rotate around={\dwangk:(0,0)}]X4);
\node[shift = {(.45, -.1)}, teal, inner sep = 2pt, fill = white] at ([rotate around={\dwangk:(0,0)}]X4) {$X_4$};
\draw[] ([rotate around={\dwangk:(0,0)}]Pg) -- ([rotate around={\dwangk:(0,0)}]X5);
\node[shift = {(.45, -.1)}, teal] at ([rotate around={\dwangk:(0,0)}]X5) {$X_5$};
\draw[] ([rotate around={\dwangk:(0,0)}]Ph) -- ([rotate around={\dwangk:(0,0)}]X6);
\node[shift = {(.45, -.1)}, teal] at ([rotate around={\dwangk:(0,0)}]X6) {$X_6$};
\draw[] ([rotate around={\dwangk:(0,0)}]Pi) -- ([rotate around={\dwangk:(0,0)}]X7);
\node[shift = {(.45, -.1)}, teal] at ([rotate around={\dwangk:(0,0)}]X7) {$X_7$};

\node[shift = {(170: 2ex)}] at ([rotate around={\dwangk:(0,0)}]Pa) {$P_1$};
\node[shift = {(150: 2ex)}] at ([rotate around={\dwangk:(0,0)}]Pb) {$P_2$};
\node[shift = {(130: 2ex)}] at ([rotate around={\dwangk:(0,0)}]Pc) {$P_3$};
\node[shift = {(110: 2ex)}] at ([rotate around={\dwangk:(0,0)}]Pd) {$P_4$};
\node[shift = {(90: 2ex)}] at ([rotate around={\dwangk:(0,0)}]Pe) {$P_5$};
\node[shift = {(70: 2ex)}] at ([rotate around={\dwangk:(0,0)}]Pf) {$P_6$};
\node[shift = {(50: 2ex)}] at ([rotate around={\dwangk:(0,0)}]Pg) {$P_7$};
\node[shift = {(30: 2ex)}] at ([rotate around={\dwangk:(0,0)}]Ph) {$P_8$};
\node[shift = {(10: 2ex)}] at ([rotate around={\dwangk:(0,0)}]Pi) {$P_9$};
\node[shift = {(10: 2ex)}] at ([rotate around={\dwangk:(0,0)}]Pj) {$P_{10}$};

\node[blue, shift = {(0, -2ex)}] at ([rotate around={\dwangk:(0,0)}]ab) {$q_2$};
\node[blue, shift = {(0, -2ex)}] at ([rotate around={\dwangk:(0,0)}]ac) {$q_3$};
\node[blue, shift = {(0, -2ex)}] at ([rotate around={\dwangk:(0,0)}]ad) {$q_4$};
\node[blue, shift = {(0, -2ex)}] at ([rotate around={\dwangk:(0,0)}]ae) {$q_5$};
\node[blue, shift = {(0, -2ex)}] at ([rotate around={\dwangk:(0,0)}]af) {$q_6$};
\node[blue, shift = {(0, -2ex)}] at ([rotate around={\dwangk:(0,0)}]ag) {$q_7$};
\node[blue, shift = {(0, -2ex)}] at ([rotate around={\dwangk:(0,0)}]ah) {$q_8$};
\node[blue, shift = {(0, -2ex)}] at ([rotate around={\dwangk:(0,0)}]ai) {$q_9$};

\draw[] ([rotate around={\dwangk:(0,0)}]Pa) -- ([rotate around={\dwangk:(0,0)}]Pj);

\draw[] ([rotate around={\dwangk:(0,0)}]Pb) -- ([rotate around={\dwangk:(0,0)}]ab);
\draw[] ([rotate around={\dwangk:(0,0)}]Pc) -- ([rotate around={\dwangk:(0,0)}]ac);
\draw[] ([rotate around={\dwangk:(0,0)}]Pd) -- ([rotate around={\dwangk:(0,0)}]ad);
\draw[] ([rotate around={\dwangk:(0,0)}]Pe) -- ([rotate around={\dwangk:(0,0)}]ae);
\draw[] ([rotate around={\dwangk:(0,0)}]Pf) -- ([rotate around={\dwangk:(0,0)}]af);
\draw[] ([rotate around={\dwangk:(0,0)}]Pg) -- ([rotate around={\dwangk:(0,0)}]ag);
\draw[] ([rotate around={\dwangk:(0,0)}]Ph) -- ([rotate around={\dwangk:(0,0)}]ah);
\draw[] ([rotate around={\dwangk:(0,0)}]Pi) -- ([rotate around={\dwangk:(0,0)}]ai);

\begin{scope}[rotate around={\dwangk:(0,0)}]
\draw[shift = {([rotate around={\dwangk:(0,0)}]ab)}] (0,0) rectangle (1.25ex, 1.25ex);
\draw[shift = {([rotate around={\dwangk:(0,0)}]ac)}] (0,0) rectangle (1.25ex, 1.25ex);
\draw[shift = {([rotate around={\dwangk:(0,0)}]ad)}] (0,0) rectangle (1.25ex, 1.25ex);
\draw[shift = {([rotate around={\dwangk:(0,0)}]ae)}] (0,0) rectangle (1.25ex, 1.25ex);
\draw[shift = {([rotate around={\dwangk:(0,0)}]af)}] (0,0) rectangle (1.25ex, 1.25ex);
\draw[shift = {([rotate around={\dwangk:(0,0)}]ag)}] (0,0) rectangle (1.25ex, 1.25ex);
\draw[shift = {([rotate around={\dwangk:(0,0)}]ah)}] (0,0) rectangle (1.25ex, 1.25ex);
\draw[shift = {([rotate around={\dwangk:(0,0)}]ai)}] (0,0) rectangle (1.25ex, 1.25ex);
\end{scope}

\draw[fill = black] ([rotate around={\dwangk:(0,0)}]Pa) circle (.25ex);
\draw[fill = black] ([rotate around={\dwangk:(0,0)}]Pb) circle (.25ex);
\draw[fill = black] ([rotate around={\dwangk:(0,0)}]Pc) circle (.25ex);
\draw[fill = black] ([rotate around={\dwangk:(0,0)}]Pd) circle (.25ex);
\draw[fill = black] ([rotate around={\dwangk:(0,0)}]Pe) circle (.25ex);
\draw[fill = black] ([rotate around={\dwangk:(0,0)}]Pf) circle (.25ex);
\draw[fill = black] ([rotate around={\dwangk:(0,0)}]Pg) circle (.25ex);
\draw[fill = black] ([rotate around={\dwangk:(0,0)}]Ph) circle (.25ex);
\draw[fill = black] ([rotate around={\dwangk:(0,0)}]Pi) circle (.25ex);
\draw[fill = black] ([rotate around={\dwangk:(0,0)}]Pj) circle (.25ex);

\draw[blue, fill = blue] ([rotate around={\dwangk:(0,0)}]ab) circle (.25ex);
\draw[blue, fill = blue] ([rotate around={\dwangk:(0,0)}]ac) circle (.25ex);
\draw[blue, fill = blue] ([rotate around={\dwangk:(0,0)}]ad) circle (.25ex);
\draw[blue, fill = blue] ([rotate around={\dwangk:(0,0)}]ae) circle (.25ex);
\draw[blue, fill = blue] ([rotate around={\dwangk:(0,0)}]af) circle (.25ex);
\draw[blue, fill = blue] ([rotate around={\dwangk:(0,0)}]ag) circle (.25ex);
\draw[blue, fill = blue] ([rotate around={\dwangk:(0,0)}]ah) circle (.25ex);
\draw[blue, fill = blue] ([rotate around={\dwangk:(0,0)}]ai) circle (.25ex);

\draw[teal, fill = teal] ([rotate around={\dwangk:(0,0)}]X2) circle (.25ex);
\draw[teal, fill = teal] ([rotate around={\dwangk:(0,0)}]X3) circle (.25ex);
\draw[teal, fill = teal] ([rotate around={\dwangk:(0,0)}]X4) circle (.25ex);
\draw[teal, fill = teal] ([rotate around={\dwangk:(0,0)}]X5) circle (.25ex);
\draw[teal, fill = teal] ([rotate around={\dwangk:(0,0)}]X6) circle (.25ex);
\draw[teal, fill = teal] ([rotate around={\dwangk:(0,0)}]X7) circle (.25ex);

}

\newcommand{\pipicm}{

 \pgfmathsetmacro{\l}{\linewidth/28.440}%
 \pgfmathsetmacro{\h}{\l/2.65}%

 \pgfmathsetmacro{\ww}{\www/2}%
 \pgfmathsetmacro{\hh}{\hhh/2}%
 \pgfmathsetmacro{\rrr}{\dwscale}%

\begin{scope}[shift = {(-.025in, -.377in)}]
\begin{scope}[shift = {(\ww in,\hh in)}]

\begin{scope}[rotate = 30]

 \draw[fill = black, black] (0, 0) circle (.4ex); 
 \node[shift = {(-.5ex, -2.5ex)}] at (0,0) {$O$};

\draw[line width = .5pt, \dwcred] (10: \rrr in) arc (10:170:\rrr in);

\begin{scope}[rotate = 10]

\foreach \a in {0,1, 3, 5, 6,7}{
\pgfmathsetmacro{\b}{160/8}
\pgfmathsetmacro{\c}{\a*\b}
\pgfmathsetmacro{\d}{\a+1}
\pgfmathsetmacro{\e}{\d*\b}
\draw[line width = .75pt, \dwcblue] (\c:\rrr in) -- (\e: \rrr in); 

}
  
\foreach \a in {2,4}{
\pgfmathsetmacro{\b}{160/8}
\pgfmathsetmacro{\c}{\a*\b}
\pgfmathsetmacro{\d}{\a+1}
\pgfmathsetmacro{\e}{\d*\b}
\draw[densely dashed, line width = .75pt, \dwcblue] (\c:\rrr in) -- (\e: \rrr in); 

}

\end{scope}

\pgfmathsetmacro{\aa}{sin(10)};
\pgfmathsetmacro{\bb}{\rrr*\aa};

\pgfmathsetmacro{\a}{1};
\pgfmathsetmacro{\aone}{90 - \a*20};
\pgfmathsetmacro{\cl}{cos(\aone)};
\pgfmathsetmacro{\lone}{\rrr*\cl};
\node[shift = {(\aone: 3ex)}] at (\aone: \rrr in) {$P_{m +1}$};

\pgfmathsetmacro{\a}{2};
\pgfmathsetmacro{\aone}{90 - \a*20};
\pgfmathsetmacro{\cl}{cos(\aone)};
\pgfmathsetmacro{\lone}{\rrr*\cl};
\node[shift = {(\aone: 3ex)}] at (\aone: \rrr in) {$P_{n -1}$};

\pgfmathsetmacro{\a}{3};
\pgfmathsetmacro{\aone}{90 - \a*20};
\pgfmathsetmacro{\cl}{cos(\aone)};
\pgfmathsetmacro{\lone}{\rrr*\cl};
\node[shift = {(\aone: 3ex)}] at (\aone: \rrr in) {$P_{n}$};

\pgfmathsetmacro{\a}{-3};
\pgfmathsetmacro{\aone}{90 - \a*20};
\pgfmathsetmacro{\cl}{cos(\aone)};
\pgfmathsetmacro{\lone}{\rrr*\cl};
\node[shift = {(\aone: 3ex)}] at (\aone: \rrr in) {$P_2$};

\pgfmathsetmacro{\a}{-2};
\pgfmathsetmacro{\aone}{90 - \a*20};
\pgfmathsetmacro{\cl}{cos(\aone)};
\pgfmathsetmacro{\lone}{\rrr*\cl};

\node[shift = {(\aone: 3ex)}] at (\aone: \rrr in) {$P_3$};

\pgfmathsetmacro{\aaa}{-1};
\pgfmathsetmacro{\aaaone}{90 - \aaa*20};
\pgfmathsetmacro{\clll}{cos(\aaaone)};
\pgfmathsetmacro{\lllone}{\rrr*\clll};

\node[shift = {(\aaaone: 3ex)}] at (\aaaone: \rrr in) {$P_{4}$};

\pgfmathsetmacro{\a}{0};
\pgfmathsetmacro{\aone}{90 - \a*20};
\pgfmathsetmacro{\cl}{cos(\aone)};
\pgfmathsetmacro{\lone}{\rrr*\cl};
\node[shift = {(\aone: 3ex)}] at (\aone: \rrr in) {$P_m$};

\pgfmathsetmacro{\aaa}{-3};
\pgfmathsetmacro{\aaaone}{90 - \aaa*20};
\pgfmathsetmacro{\clll}{cos(\aaaone)};
\pgfmathsetmacro{\lllone}{\rrr*\clll};

\pgfmathsetmacro{\a}{3};
\pgfmathsetmacro{\aone}{90 - \a*20};
\pgfmathsetmacro{\cl}{cos(\aone)};
\pgfmathsetmacro{\lone}{\rrr*\cl};

\pgfmathsetmacro{\ea}{3};
\pgfmathsetmacro{\eaone}{90 - \ea*20};
\pgfmathsetmacro{\ecl}{cos(\eaone)};
\pgfmathsetmacro{\elone}{\rrr*\ecl};

\pgfmathsetmacro{\fa}{5};
\pgfmathsetmacro{\faone}{90 - \fa*20};
\pgfmathsetmacro{\fcl}{cos(\faone)};
\pgfmathsetmacro{\flone}{\rrr*\fcl};

\draw[\dwcblue] (170: \rrr in) -- (\flone in, \bb in);

\pgfmathsetmacro{\a}{1};
\pgfmathsetmacro{\aone}{90 - \a*20};
\pgfmathsetmacro{\cl}{cos(\aone)};
\pgfmathsetmacro{\lone}{\rrr*\cl};

\pgfmathsetmacro{\ka}{sin(50)};
\pgfmathsetmacro{\cka}{cos(50)};
\pgfmathsetmacro{\kb}{sin(70)};
\pgfmathsetmacro{\kc}{sin(10)};
\pgfmathsetmacro{\kd}{cos(10)};
\pgfmathsetmacro{\ld}{\rrr*cos(10)};
\pgfmathsetmacro{\ldd}{\rrr - \ld};

\pgfmathsetmacro{\ja}{2*\rrr*\ka}; 
\pgfmathsetmacro{\jb}{\rrr*\kb}; 
\pgfmathsetmacro{\jc}{\rrr*\kc}; 
\pgfmathsetmacro{\jd}{\jb-\jc}; 
\pgfmathsetmacro{\je}{\jd/\ja}; 
\pgfmathsetmacro{\aangle}{asin(\je)}; 

\pgfmathsetmacro{\lengthoffset}{\ja  - \ld};

\draw[line width = .5pt, \dwcteal] (\lengthoffset in, \bb in) arc (0:\aangle:\ja in);

\pgfmathsetmacro{\wa}{cos(\aangle)};
\pgfmathsetmacro{\wb}{\ja*\wa};
\pgfmathsetmacro{\wc}{\wb + \ldd-\rrr};

\begin{scope}[shift = {(270: 13ex)}]
\draw[dashed, shift = {(\wc in, \bb in)}] (0, 5ex) -- ( 0, 12.5ex);

\pgfmathsetmacro{\za}{\wc + \ld};
\pgfmathsetmacro{\avex}{\za/2};

\pgfmathsetmacro{\zza}{\bb + \jc};
\pgfmathsetmacro{\avey}{\zza/2};

\draw[shift = {(\wc in, \bb in)}] (0, 0ex) -- ( 0, 4ex);
\draw[shift = {(10: \rrr in)}] (0, 0ex) -- ( 0, 4ex);
\draw[shift = {(270: 0ex)}] (10: \rrr in) -- (\wc in, \bb in);
\node[fill = white] at (\avex in, \avey in) {$d$};

\end{scope}

\begin{scope}[shift = {(270: 18ex)}]

\pgfmathsetmacro{\za}{\lengthoffset + \ld};
\pgfmathsetmacro{\avex}{\za/2};

\pgfmathsetmacro{\zza}{\bb + \jc};
\pgfmathsetmacro{\avey}{\zza/2};

\draw[shift = {(\lengthoffset in, \bb in)}] (0, 0ex) -- ( 0, 4ex);
\draw[dashed, shift = {(\lengthoffset in, \bb in)}] (0, 5ex) -- ( 0, 17ex);
\draw[shift = {(10: \rrr in)}] (0, 0ex) -- ( 0, 4ex);
\draw[shift = {(270: 0ex)}] (10: \rrr in) -- (\lengthoffset in, \bb in);

\node[fill = white] at (\avex in, \avey in) {$s$};

\end{scope}

\begin{scope}[shift = {(270: 8ex)}]
\draw[shift = {(10: \rrr in)}] (0, 0ex) -- ( 0, 4ex);
\draw[shift = {(170: \rrr in)}] (0, 0ex) -- ( 0, 4ex);
\draw[shift = {(0, 0)}] (0, 4ex) -- ( 0, 2ex);

\node[] at (0, 0ex) {$l_n$};

\draw[shift = {(270: 0ex)}] (170: \rrr in) -- (10: \rrr in);
\end{scope}

\draw[fill = black] (170: \rrr in) circle (.25ex);   
\node[inner sep = 1pt, fill = white, shift = {(245: 3ex)}] at (170: \rrr in) {$P_1 = q_1$};   

\draw[fill = black] (10: \rrr in) circle (.25ex);   
\node[shift = {(55: 3ex)}] at (10: \rrr in) {$P_{n+1}$};

\foreach \a in {1,2,3,4}{
\pgfmathsetmacro{\aone}{90 - \a*20};
\pgfmathsetmacro{\cl}{cos(\aone)};
\pgfmathsetmacro{\lone}{\rrr*\cl};

\draw[fill = black] (\aone: \rrr in) circle (.25ex);

}

\pgfmathsetmacro{\aone}{90 - 20};
\pgfmathsetmacro{\cl}{cos(\aone)};
\pgfmathsetmacro{\lone}{\rrr*\cl};

\draw[] (\aone: \rrr in) -- (\lone in, \bb in);

\draw[] (\lone in, \bb in) rectangle (\lone in - 1.75ex, \bb in + 1.75ex);

\draw[fill = black] (\lone in, \bb in) circle (.25ex);

\draw[] (\aone: \rrr in) -- (170: \rrr in);

\foreach \a in {-3, -2, -1, 0}{
\pgfmathsetmacro{\aone}{90 - \a*20};
\pgfmathsetmacro{\cl}{cos(\aone)};
\pgfmathsetmacro{\lone}{\rrr*\cl};

\draw[fill = black] (\aone: \rrr in) circle (.25ex);

}

\node[fill = white] at (120:\rrr *\cka in) {$l_m$};

\node[inner sep = 1pt, fill = white, shift = {(260: 3ex)}] at (\wc in, \bb in) {$q_{m+1}$};

\draw[fill = black] (\lengthoffset in, \bb in) circle (.25ex);

\node[inner sep = 1pt, fill = white, shift = {(-35: 2.5ex)}] at (\lengthoffset in, \bb in) {$Q$};

\end{scope}

\end{scope}
\end{scope}

}

\newcommand{\pipicn}{

 \pgfmathsetmacro{\l}{\linewidth/28.440}%
 \pgfmathsetmacro{\h}{\l/2.65}%

 \pgfmathsetmacro{\ww}{\www/2}%
 \pgfmathsetmacro{\hh}{\hhh/2}%
 \pgfmathsetmacro{\rrr}{\dwscale}%

\pgfmathsetmacro{\newr}{1/cos(55)}
\pgfmathsetmacro{\newrr}{1/cos(44)}
\pgfmathsetmacro{\news}{1/cos(33)}
\pgfmathsetmacro{\newt}{1/cos(22)}
\pgfmathsetmacro{\newu}{1/cos(11)}

\pgfmathsetmacro{\newuu}{tan(44)}
\pgfmathsetmacro{\newtt}{tan(33)}
\pgfmathsetmacro{\newss}{tan(22)}
\pgfmathsetmacro{\newrrr}{tan(11)}

\pgfmathsetmacro{\hyp}{\rrr*\newr}
\pgfmathsetmacro{\hypp}{\rrr*\newrr}
\pgfmathsetmacro{\hyppp}{\rrr*\newrrr}
\pgfmathsetmacro{\hyps}{\rrr*\news}
\pgfmathsetmacro{\hypt}{\rrr*\newt}
\pgfmathsetmacro{\hypss}{\rrr*\newss}
\pgfmathsetmacro{\hyptt}{\rrr*\newtt}
\pgfmathsetmacro{\hypu}{\rrr*\newu}
\pgfmathsetmacro{\hypuu}{\rrr*\newuu}
\begin{scope}[shift = {(.755in, -2.738in)}]

\begin{scope}[shift = {(\ww in,\hh in)}]

\begin{scope}[rotate = 10]

\draw[shorten <=0, shorten >=0, black] (73: \hyps in)--(95: \hyp in);      
\draw[shorten <=0, shorten >=0, black] (117: \hyps in)--(95: \hyp in);

\draw[black, shorten <= 0 in] (106: \hypp in) -- (84: \hypt in);
\draw[black, shorten <=0 in] (117: \hyps in) -- (95: \hypu in); 
\draw[black, shorten <= 0 in]  (73: \hyps in) -- (95: \hypu in);
\draw[black!50!black, shorten <=0 in] (106: \hypt in) -- (84: \hypp in);

\draw[] (95: \hyps in) -- (117: \hyps in);

\pgfmathsetmacro{\xpnno}{cos(95)}
\pgfmathsetmacro{\ypnno}{sin(95)}
\pgfmathsetmacro{\xdpnno}{\hyps*\xpnno}
\pgfmathsetmacro{\ydpnno}{\hyps*\ypnno}

\pgfmathsetmacro{\xpno}{cos(106)}
\pgfmathsetmacro{\ypno}{sin(106)}
\pgfmathsetmacro{\xdpno}{\hypp*\xpno}
\pgfmathsetmacro{\ydpno}{\hypp*\ypno}

\pgfmathsetmacro{\xpnnt}{cos(106)}
\pgfmathsetmacro{\ypnnt}{sin(106)}
\pgfmathsetmacro{\xdpnnt}{\hypt*\xpnnt}
\pgfmathsetmacro{\ydpnnt}{\hypt*\ypnnt}

\pgfmathsetmacro{\xpnt}{cos(117)}
\pgfmathsetmacro{\ypnt}{sin(117)}
\pgfmathsetmacro{\xdpnt}{\hyps*\xpnt}
\pgfmathsetmacro{\ydpnt}{\hyps*\ypnt}

\pgfmathsetmacro{\xdaveone}{\xdpnt/2 + \xdpnno/2}
\pgfmathsetmacro{\ydaveone}{\ydpnt/2 + \ydpnno/2}
\draw[rotate around = {16: (\xdaveone in, \ydaveone in)}] (\xdaveone in, \ydaveone in) rectangle (\xdaveone in +2ex, \ydaveone in + 2ex);

\draw[\dwcred] (106: \hypp in) -- (\xdaveone in, \ydaveone in);

\draw[black] (106: \hypt in) -- (\xdaveone in, \ydaveone in);

\draw[fill = black] (\xdaveone in, \ydaveone in) circle (.25ex);
\node[shift = {(170: 2.5ex)}, fill = white, inner sep = 1pt] at (\xdaveone in, \ydaveone in)  {$A$};

\pgfmathsetmacro{\xdiffone}{\xdpno - \xdaveone}
\pgfmathsetmacro{\ydiffone}{\ydpno - \ydaveone}
\pgfmathsetmacro{\xdoublediffone}{2*\xdiffone}
\pgfmathsetmacro{\ydoublediffone}{2*\ydiffone}
\pgfmathsetmacro{\xshiftdoublediffone}{\xdpnno +\xdoublediffone}
\pgfmathsetmacro{\yshiftdoublediffone}{\ydpnno +\ydoublediffone}

\draw[rotate around = {16: (\xdpnno in, \ydpnno in)}] (\xdpnno in, \ydpnno in) rectangle (\xdpnno in -2ex, \ydpnno in + 2ex);

\draw[\dwcred] (\xdpnno in, \ydpnno in) -- (\xshiftdoublediffone in, \yshiftdoublediffone in);
\draw[fill = black] (\xshiftdoublediffone in, \yshiftdoublediffone in) circle (.25ex);

\node[shift = {(170: 2.5ex)}, fill = white, inner sep = 1pt] at (\xshiftdoublediffone in, \yshiftdoublediffone in)  {$B$};

\draw[\dwcblue] (95: .9*\rrr in) -- (95: \hyp in);
\draw[\dwcblue] (95: .6*\rrr in) -- (95: .75*\rrr in);
\draw[dashed,\dwcblue] (95: .9*\rrr in) -- (95: .75*\rrr in);

\draw[black, fill = black] (73: \hyps in) circle (.2ex);  
\node[shift = {(4ex, 1ex)}, fill = white, inner sep = 1.5pt] at (73: \hyps in) {$P_{i,3}$};

\draw[black, fill = black] (84: \hypt in) circle (.2ex);  
\node[shift = {(300:3ex)}, fill = white, inner sep = 1pt] at (84: \hypt in) {$P_{i-1,3}$};

\draw[black, fill = black] (84: \hypp in) circle (.2ex);  
\node[shift = {(4ex, 1ex)}, fill = white, inner sep = 1.5pt] at (84: \hypp in) {$P_{i,2}$};

\draw[black, fill = black] (95: \hypu in) circle (.2ex);  

\draw[black, fill = black] (95: \hyps in) circle (.25ex);  
\node[font = \scriptsize, shift = {(10:4.5ex)}, fill = white, inner sep = 1pt] at (95: \hyps in) {$P_{2, i-1}$};

\draw[black, fill = black] (106: \hypt in) circle (.2ex);  
\node[shift = {(260: 2ex)}, fill = white, inner sep = 1pt] at (106: \hypt in) {$P_{2, i-2}$};

\draw[black, fill = black] (106: \hypp in) circle (.25ex);  
\node[shift = {(-4ex,1ex)}, fill = white, inner sep = 1.5pt] at (106: \hypp in) {$P_{1,i-1}$};

\draw[black, fill = black] (117: \hyps in) circle (.2ex);  
\node[shift = {(-4ex,1ex)}, fill = white, inner sep = 1.5pt] at (117: \hyps in) {$P_{1, i-2}$};

\draw[black, fill = black] (95: .6*\rrr in) circle (.4ex);  
\node[shift = {(2ex, 0)}] at (95: .6*\rrr in) {$O$};

\draw[black, fill = black] (95: \hyp in) circle (.3ex);  
\node[shift = {(-4ex, 0ex)}] at (95: \hyp in) {$P_{1,i}$};
\node[shift = {(4ex, 1ex)}] at (95: \hyp in) {$P_{1,i}$};

\end{scope}

\end{scope}

\end{scope}

}

\newcommand{\pipico}{

 \pgfmathsetmacro{\l}{\linewidth/28.440}%
 \pgfmathsetmacro{\h}{\l/2.65}%

 \pgfmathsetmacro{\ww}{\www/2}%
 \pgfmathsetmacro{\hh}{\hhh/2}%
 \pgfmathsetmacro{\rrr}{\dwscale}%

\pgfmathsetmacro{\newr}{1/cos(55)}
\pgfmathsetmacro{\news}{1/cos(25)}
\pgfmathsetmacro{\hyps}{\rrr*\news}

\pgfmathsetmacro{\aaa}{cos(55)}
\pgfmathsetmacro{\aaaa}{cos(25)}

\pgfmathsetmacro{\bbb}{\aaa*\rrr}
\pgfmathsetmacro{\bbbb}{\aaaa*\rrr}

\pgfmathsetmacro{\hyp}{\rrr*\newr}

\pgfmathsetmacro{\newrr}{1/cos(40)}
\pgfmathsetmacro{\newrrr}{tan(15)}
\pgfmathsetmacro{\hypp}{\rrr*\newrr}
\pgfmathsetmacro{\hyppp}{\rrr*\newrrr}
\pgfmathsetmacro{\LC}{\hyppp + .5}

\pgfmathsetmacro{\newu}{1/cos(15)}
\pgfmathsetmacro{\newuu}{tan(40)}
\pgfmathsetmacro{\hypu}{\rrr*\newu}
\pgfmathsetmacro{\hypuu}{\rrr*\newuu}
\pgfmathsetmacro{\LB}{\hypuu + .5}

\begin{scope}[shift = {(.405 in, -.96 in)}]

\begin{scope}[shift = {(\ww in,\hh in)}]

\begin{scope}[rotate = 25]

\draw[line width = .75pt, \dwcblue] (20: \rrr in) arc (20:170:\rrr in);

\draw[shorten <=-30, shorten >=-20, black] (40: \rrr in)--(95: \hyp in);      
\draw[shorten <=-30, shorten >=-20, black] (150: \rrr in)--(95: \hyp in);

\draw[\dwcteal] (0,0)--(150:\rrr in);
\draw[\dwcteal] (0,0)--(40:\rrr in);

\draw[\dwcteal, shorten >= -0in] (0,0) -- (95: \hyp in);

\draw[black, fill = black] (0, 0) circle (.4ex);  
\node[shift = {(-20:3ex)}, fill = white, inner sep = 1pt] at (0,0) {$O$};

\draw[black, fill = black] (40: \rrr in) circle (.3ex);  
\node[shift = {(20: 4ex)}, fill = white, inner sep = 1pt] at (40: \rrr in) {$P_{3,3}$};
\node[shift = {(200: 4ex)}, fill = white, inner sep = 1pt] at (40: \rrr in) {$P_{3}$};

\pgfmathsetmacro{\recipcosofsixtytwopointfive}{1/cos(27.5)};
\pgfmathsetmacro{\cosofsixtytwopointfive}{cos(27.5)};
\pgfmathsetmacro{\sinofsixtytwopointfive}{sin(27.5)};

\pgfmathsetmacro{\lengthcosofsixtytwopointfive}{\rrr*\cosofsixtytwopointfive};
\pgfmathsetmacro{\lengthrecipcosofsixtytwopointfive}{\rrr*\recipcosofsixtytwopointfive};
\pgfmathsetmacro{\aa}{\lengthrecipcosofsixtytwopointfive - \lengthcosofsixtytwopointfive};
\pgfmathsetmacro{\a}{2*\aa};
\pgfmathsetmacro{\x}{cos(122.5)};
\pgfmathsetmacro{\y}{sin(122.5)};
\pgfmathsetmacro{\z}{cos(95)};
\pgfmathsetmacro{\w}{sin(95)};
\pgfmathsetmacro{\lz}{\rrr*\z};
\pgfmathsetmacro{\lw}{\rrr*\w};

\pgfmathsetmacro{\kz}{\lengthcosofsixtytwopointfive*\x};
\pgfmathsetmacro{\kw}{\lengthcosofsixtytwopointfive*\y};

\pgfmathsetmacro{\nx}{\x*\a};
\pgfmathsetmacro{\ny}{\y*\a};

\pgfmathsetmacro{\mx}{\nx + \lz};
\pgfmathsetmacro{\my}{\ny + \lw};

\draw[] (0,0) -- (67.5:\lengthrecipcosofsixtytwopointfive in);

\draw[] (0,0) -- (122.5:\lengthcosofsixtytwopointfive in);

\draw[] (122.5:\lengthrecipcosofsixtytwopointfive in) -- (67.5:\lengthrecipcosofsixtytwopointfive in);

\draw[\dwcviolet] (150: \rrr in) -- (95: \rrr in);

\node[white, thick, rotate = 57.5, anchor = south east, rotate = 0, draw,minimum width=2ex,minimum height=2ex] at (\lz in +.35pt, \lw in -.07pt) {};

\node[rotate = 57.5, anchor = south east, rotate = 0, draw,thin,minimum width=2ex,minimum height=2ex] at (\lz in +.27pt, \lw in -.07pt) {};

\node[white, thick, rotate = 57.5, anchor = south west, rotate = 0, draw,minimum width=2ex,minimum height=2ex] at (\kz in -.15pt, \kw in -.35pt) {};

\node[rotate = 57.5, anchor = south west, rotate = 0, draw,thin,minimum width=2ex,minimum height=2ex] at (\kz in + 0pt, \kw in -.35pt) {};

\node[shift = {(75: 2.5ex)}] at (95: \rrr in) {$P_2$};

\node[inner sep = 5pt, shift = {(25: 4ex)}] at (95: \hyp in) {$P_{3,1}$};
\node[inner sep = 1pt, fill = white, shift = {(205: 4ex)}] at (95: \hyp in) {$P_{1,3}$};

\draw[\dwcred] (95:\rrr in) -- (\mx in, \my in);

\draw[fill = black] (\mx in, \my in) circle (.3ex);  
\node[shift = {(205: 4ex)}] at (\mx in, \my in) {$B$};

\draw[\dwcred] (122.5:\lengthcosofsixtytwopointfive in) -- (122.5:\lengthrecipcosofsixtytwopointfive in);

\draw[fill = black] (67.5:\lengthrecipcosofsixtytwopointfive in) circle (.2ex);
\node[shift = {(20: 4ex)}] at (67.5:\lengthrecipcosofsixtytwopointfive in) {$P_{3,2}$};

\draw[fill = black] (122.5:\lengthrecipcosofsixtytwopointfive in) circle (.3ex);
\node[shift = {(205: 4ex)}] at (122.5:\lengthrecipcosofsixtytwopointfive in) {$P_{1,2}$};

\draw[fill = black] (122.5:\lengthcosofsixtytwopointfive in) circle (.3ex);
\node[shift = {(20: 3ex)}] at (122.5:\lengthcosofsixtytwopointfive in) {$M$};

\draw[fill = black] (95: \rrr in) circle (.3ex);  

\draw[black, fill = black] (150: \rrr in) circle (.3ex);  
\node[shift = {(205: 4ex)}] at (150: \rrr in) {$P_{1,1}$};
\node[shift = {(25: 4ex)}] at (150: \rrr in) {$P_{1}$};

\draw[fill = black] (95: \hyp in) circle (.3ex);  

\end{scope}

\end{scope}

\end{scope}

}

\newcommand{\pipicp}{

 \pgfmathsetmacro{\l}{\linewidth/28.440}%
 \pgfmathsetmacro{\h}{\l/2.65}%

 \pgfmathsetmacro{\ww}{\www/2}%
 \pgfmathsetmacro{\hh}{\hhh/2}%
 \pgfmathsetmacro{\rrr}{\dwscale/2}%
 
\begin{scope}[shift = {(.035 in, -.488 in)}]

\begin{scope}[shift = {(\ww in,\hh in)}]

\begin{scope}[rotate = 0]

\draw[fill = black] (0,0) circle (.4ex);  
\node[shift = {(-45: 3ex)}] at (0,0) {$O$};

\draw[dashed, line width = .75pt, \dwcblue] (215: \rrr in) arc (215:200:\rrr in);
\draw[line width = .75pt, \dwcblue] (-30: \rrr in) arc (-30:200:\rrr in);
\draw[dashed, line width = .75pt, \dwcblue] (-30: \rrr in) arc (-30:-45:\rrr in);

\node[shift = {(100: -4ex)}] at (100:\rrr in) {$ P_1 = P_{n+1}$};

\node[shift = {(30: -3ex)}] at (30:\rrr in) {$P_{n}$};

\node[shift = {(10: -3ex)}] at (10:\rrr in) {$P_{n-1}$};

\node[shift = {(-20: -3ex)}] at (-20:\rrr in) {$P_{n-2}$};

\node[shift = {(190: -4ex)}] at (190:\rrr in) {$P_4$};

\node[shift = {(170: -4ex)}] at (170:\rrr in) {$P_3$};

\node[shift = {(140: -3ex)}] at (140:\rrr in) {$P_2$};

 \pgfmathsetmacro{\aa}{cos(35)}%
  \pgfmathsetmacro{\ra}{1/\aa}%
   \pgfmathsetmacro{\la}{\rrr*\ra}%

\node[shift = {(65: 3ex)}] at (65: \la in) {$V_n$};

 \pgfmathsetmacro{\ab}{cos(10)}%
  \pgfmathsetmacro{\rb}{1/\ab}%
   \pgfmathsetmacro{\lb}{\rrr*\rb}%

\node[shift = {(20: 3ex)}] at (20: \lb in) {$V_{n-1}$};

\draw[\dwcred] (65: \la in) -- (20: \lb in);

 \pgfmathsetmacro{\ac}{cos(15)}%
  \pgfmathsetmacro{\rc}{1/\ac}%
   \pgfmathsetmacro{\lc}{\rrr*\rc}%

\node[shift = {(-5: 3ex)}] at (-5: \lc in) {$V_{n-2}$};

\draw[\dwcred] (20: \lb in) -- (-5: \lc in);

\draw[\dwcred] (-5: \lc in) -- (-20: \rrr in);

 \pgfmathsetmacro{\ad}{cos(10)}%
  \pgfmathsetmacro{\rd}{1/\ad}%
   \pgfmathsetmacro{\ld}{\rrr*\rd}%

\node[inner sep = 0pt, shift = {(180: 3ex)}] at (180: \ld in) {$V_{3}$};

\draw[\dwcred] (190:\rrr in) -- (180: \ld in);  

 \pgfmathsetmacro{\ae}{cos(15)}%
  \pgfmathsetmacro{\re}{1/\ae}%
   \pgfmathsetmacro{\le}{\rrr*\re}%

\node[shift = {(155: 3ex)}] at (155: \le in) {$V_{2}$};

\draw[\dwcred] (180: \ld in) -- (155: \le in);
\draw[\dwcred] (180: \ld in) -- (155: \le in);

 \pgfmathsetmacro{\af}{cos(20)}%
  \pgfmathsetmacro{\rf}{1/\af}%
   \pgfmathsetmacro{\lf}{\rrr*\rf}%

\node[shift = {(120: 3ex)}] at (120: \lf in) {$V_1$};

\draw[\dwcred] (155: \le in) -- (120: \lf in);

\draw[\dwcred] (120: \lf in) -- (65:\la in);

\draw[fill = black] (120: \lf in) circle (.25ex); 
\draw[] (100:\rrr in) -- (30:\rrr in);  
\draw[] (30:\rrr in) -- (10:\rrr in);  
\draw[] (10:\rrr in) -- (-20:\rrr in);  
\draw[dashed, shorten >= .9in] (-20:\rrr in) -- (270:\rrr in);  
\draw[dashed, shorten >= .75in] (190:\rrr in) -- (245: \rrr in);  
\draw[] (170:\rrr in) -- (190:\rrr in);  
\draw[] (140:\rrr in) -- (170:\rrr in);  
\draw[] (100:\rrr in) -- (140:\rrr in);  

\draw[fill = black] (100:\rrr in) circle (.3ex);
\draw[fill = black] (30:\rrr in) circle (.3ex); 
\draw[fill = black] (10:\rrr in) circle (.3ex); 
\draw[fill = black] (-20:\rrr in) circle (.3ex); 
\draw[fill = black] (190:\rrr in) circle (.3ex);  
\draw[fill = black] (170:\rrr in) circle (.3ex);  
\draw[fill = black] (140:\rrr in) circle (.3ex);  
 \draw[fill = black] (65: \la in) circle (.25ex);  
\draw[fill = black] (20: \lb in) circle (.25ex); 
\draw[fill = black] (-5: \lc in) circle (.25ex); 
\draw[fill = black] (180: \ld in) circle (.25ex); 
\draw[fill = black] (155: \le in) circle (.25ex);

\end{scope}

\end{scope}

\end{scope}
}

\newcommand{\pipicr}{

 \pgfmathsetmacro{\l}{\linewidth/28.440}%
 \pgfmathsetmacro{\h}{\l/2.65}%

 \pgfmathsetmacro{\ww}{\www/2}%
 \pgfmathsetmacro{\hh}{\hhh/2}%
 \pgfmathsetmacro{\rrr}{\dwscale}%

\pgfmathsetmacro{\newr}{1/cos(55)}
\pgfmathsetmacro{\newrr}{1/cos(44)}
\pgfmathsetmacro{\news}{1/cos(33)}
\pgfmathsetmacro{\newt}{1/cos(22)}
\pgfmathsetmacro{\newu}{1/cos(11)}

\pgfmathsetmacro{\newuu}{tan(44)}
\pgfmathsetmacro{\newtt}{tan(33)}
\pgfmathsetmacro{\newss}{tan(22)}
\pgfmathsetmacro{\newrrr}{tan(11)}

\pgfmathsetmacro{\hyp}{\rrr*\newr}
\pgfmathsetmacro{\hypp}{\rrr*\newrr}
\pgfmathsetmacro{\hyppp}{\rrr*\newrrr}
\pgfmathsetmacro{\hyps}{\rrr*\news}
\pgfmathsetmacro{\hypt}{\rrr*\newt}
\pgfmathsetmacro{\hypss}{\rrr*\newss}
\pgfmathsetmacro{\hyptt}{\rrr*\newtt}
\pgfmathsetmacro{\hypu}{\rrr*\newu}
\pgfmathsetmacro{\hypuu}{\rrr*\newuu}
\pgfmathsetmacro{\ax}{\rrr*cos(106)}
\pgfmathsetmacro{\ay}{\rrr*sin(106)}
\pgfmathsetmacro{\bx}{\hypp*cos(150)}
\pgfmathsetmacro{\by}{\hypp*sin(150)}

\pgfmathsetmacro{\xsquare}{(\ax - \bx)*(\ax-\bx)}
\pgfmathsetmacro{\ysquare}{(\ay - \by)*(\ay-\by)}
\pgfmathsetmacro{\radius}{sqrt(\xsquare +\ysquare)}

\begin{scope}[shift = {(.128in, -1.29in)}]

\begin{scope}[shift = {(\ww in,\hh in)}]

\begin{scope}[rotate = 10]

\draw[line width = .75pt, \dwcblue] (20: \rrr in) arc (20:170:\rrr in);

\draw[shorten <=-30, shorten >=-20, black] (40: \rrr in)--(95: \hyp in);      
\draw[shorten <=-30, shorten >=-20, black] (150: \rrr in)--(95: \hyp in);

\draw[\dwcviolet, shorten <=-\hyppp in] (62: \rrr in) -- (106:\hypp in);

\draw[\dwcorange] (106:\hypp in) ++(240:\radius in) arc (240:330:\radius in);

\draw[\dwcred, shorten <=-\hypss in] (84: \rrr in) -- (117:\hyps in);

\draw[\dwcteal] (0,0) -- (150: \rrr in);
\draw[\dwcteal] (0,0) -- (84: \rrr in);
\draw[\dwcteal] (0,0) -- (62: \rrr in);
\draw[\dwcteal] (0,0) -- (40: \rrr in);

\draw[black, fill = black] (51: \hypu in) circle (.2ex);  
\node[shift = {(3ex, 1ex)}, fill = white, inner sep = .5pt] at (51: \hypu in) {$P_{A,Q}$};

\draw[black, fill = black] (62: \hypt in) circle (.2ex);  
\node[shift = {(3ex, 1ex)}, fill = white, inner sep = .5pt, rotate = 0] at (62: \hypt in) {$P_{A,B}$};

\draw[black, fill = black] (73: \hypu in) circle (.2ex);  

\draw[black, fill = black] (106: \hypp in) circle (.3ex);  
\node[shift = {(-4ex,1ex)}, fill = white, inner sep = 1.5pt] at (106: \hypp in) {$P$};

\draw[black, fill = black] (117: \hyps in) circle (.3ex);  
\node[shift = {(-4ex,1ex)}, fill = white, inner sep = 1.5pt] at (117: \hyps in) {$P_{B,C}$};

\draw[black, fill = black] (0, 0) circle (.4ex);  
\node[shift = {(2ex, 0)}] at (0,0) {$O$};

\draw[black, fill = black] (40: \rrr in) circle (.3ex);  
\node[shift = {(85:-2.5ex)}, fill = white, inner sep = 1.5pt] at (40: \rrr in) {$A$};

\node[rotate = 0, anchor = north east ,minimum width=2ex,minimum height=2ex, shift = {(55:1.2)}] at (40: \rrr in) {};

\draw[black, fill = black] (62: \rrr in) circle (.3ex);  
\node[shift = {(107: -2.5ex)}, fill = white, inner sep = 1.5pt] at (62: \rrr in) {$Q$};

\node[rotate= -18, anchor = north east, draw,thin,minimum width=2ex,minimum height=2ex] at (62: \rrr in) {};

\draw[black, fill = black] (84: \rrr in) circle (.3ex);  
\node[shift = {(129: -2.5ex)}, fill = white, inner sep = 1.5pt] at (84: \rrr in) {$B$};

\node[rotate = 4, anchor = north east, draw,thin,minimum width=2ex,minimum height=2ex] at (84:\rrr in) {};

\draw[black, fill = black] (150: \rrr in) circle (.3ex);  
\node[shift = {(195:-2.5ex)}, fill = white, inner sep = 1.5pt] at (150: \rrr in) {$C$};

\node[rotate = 70, anchor = north east, draw,thin,minimum width=2ex,minimum height=2ex] at (150: \rrr in) {};

\draw[black, fill = black] (95: \hyp in) circle (.3ex);  
\node[shift = {(-4ex,1ex)}] at (95: \hyp in) {$P_{A,C}$};

\end{scope}

\end{scope}

\end{scope}
}

\newcommand{\dwpice}{
\renewcommand{\dwscale}{1.0175}
\noindent\begin{minipage}{\linewidth}
\begin{center}
\begin{tikzpicture}
\pipice
\end{tikzpicture}
\end{center}
\vskip .125in
\begin{center}
Figure \refstepcounter{dwfig}\label{e}\arabic{dwfig}
\end{center}
\end{minipage}
}

\newcommand{\dwpicf}{
\renewcommand{\dwscale}{1.0485}
\noindent\begin{minipage}{\linewidth}
\begin{center}
\begin{tikzpicture}
\pipicf
\end{tikzpicture}
\end{center}
\vskip .125in
\begin{center}
Figure \refstepcounter{dwfig}\label{f}\arabic{dwfig}
\end{center}
\end{minipage}
}

\newcommand{\dwpicg}{
\renewcommand{\dwscale}{1.191}
\noindent\begin{minipage}{\linewidth}
\begin{center}
\begin{tikzpicture}
\pipicg
\end{tikzpicture}
\end{center}
\vskip .125in
\begin{center}
Figure \refstepcounter{dwfig}\label{g}\arabic{dwfig}
\end{center}
\end{minipage}
}

\newcommand{\dwpich}{
\renewcommand{\dwscale}{.9285}
\noindent\begin{minipage}{\linewidth}
\begin{center}
\begin{tikzpicture}
\pipich
\end{tikzpicture}
\end{center}
\vskip .125in
\begin{center}
Figure \refstepcounter{dwfig}\label{h}\arabic{dwfig}
\end{center}
\end{minipage}
}

\newcommand{\dwpichh}{
\renewcommand{\dwscale}{.9285}
\noindent\begin{minipage}{\linewidth}
\begin{center}
\begin{tikzpicture}
\pipichh
\end{tikzpicture}
\end{center}
\vskip .125in
\begin{center}
Figure \refstepcounter{dwfig}\label{hh}\arabic{dwfig}
\end{center}
\end{minipage}
}

\newcommand{\dwpici}{
\renewcommand{\dwscale}{1.4275}
\noindent\begin{minipage}{\linewidth}
\begin{center}
\begin{tikzpicture}
\pipici
\end{tikzpicture}
\end{center}
\vskip .125in
\begin{center}
Figure \refstepcounter{dwfig}\label{i}\arabic{dwfig}
\end{center}
\end{minipage}
}

\newcommand{\dwpicj}{
\renewcommand{\dwscale}{1.1908}
\noindent\begin{minipage}{\linewidth}
\begin{center}
\begin{tikzpicture}
\pipicj
\end{tikzpicture}
\end{center}
\vskip .125in
\begin{center}
Figure \refstepcounter{dwfig}\label{j}\arabic{dwfig}
\end{center}
\end{minipage}
}

\newcommand{\dwpick}{
\renewcommand{\dwscale}{1.3787}
\noindent\begin{minipage}{\linewidth}
\begin{center}
\begin{tikzpicture}
\pipicnewk
\end{tikzpicture}
\end{center}
\vskip .125in
\begin{center}
Figure \refstepcounter{dwfig}\label{k}\arabic{dwfig}
\end{center}
\end{minipage}
}

\newcommand{\dwpicl}{
\renewcommand{\dwscale}{1.3668}
\noindent\begin{minipage}{\linewidth}
\begin{center}
\begin{tikzpicture}
\pipicnewl
\end{tikzpicture}
\end{center}
\vskip .125in
\begin{center}
Figure \refstepcounter{dwfig}\label{l}\arabic{dwfig}
\end{center}
\end{minipage}
}

\newcommand{\dwpicm}{
\renewcommand{\dwscale}{1.1488}
\noindent\begin{minipage}{\linewidth}
\begin{center}
\begin{tikzpicture}
\pipicm
\end{tikzpicture}
\end{center}
\vskip .125in
\begin{center}
Figure \refstepcounter{dwfig}\label{m}\arabic{dwfig}
\end{center}
\end{minipage}
}

\newcommand{\dwpicn}{
\renewcommand{\dwscale}{1.8175}
\noindent\begin{minipage}{\linewidth}
\begin{center}
\begin{tikzpicture}
\pipicn
\end{tikzpicture}
\end{center}
\vskip .125in
\begin{center}
Figure \refstepcounter{dwfig}\label{n}\arabic{dwfig}
\end{center}
\end{minipage}
}

\newcommand{\dwpico}{
\renewcommand{\dwscale}{1.1575}
\noindent\begin{minipage}{\linewidth}
\begin{center}
\begin{tikzpicture}
\pipico
\end{tikzpicture}
\end{center}
\vskip .125in
\begin{center}
Figure \refstepcounter{dwfig}\label{o}\arabic{dwfig}
\end{center}
\end{minipage}
}

\newcommand{\dwpicp}{
\renewcommand{\dwscale}{2.224}
\noindent\begin{minipage}{\linewidth}
\begin{center}
\begin{tikzpicture}
\renewcommand{\dwscale}{2.224}
\pipicp
\end{tikzpicture}
\end{center}
\vskip .125in
\begin{center}
Figure \refstepcounter{dwfig}\label{p}\arabic{dwfig}
\end{center}
\end{minipage}
}

\newcommand{\dwpicr}{
\renewcommand{\dwscale}{1.2786}
\noindent\begin{minipage}{\linewidth}
\begin{center}
\begin{tikzpicture}
\pipicr
\end{tikzpicture}
\end{center}
\vskip .125in
\begin{center}
Figure \refstepcounter{dwfig}\label{r}\arabic{dwfig}
\end{center}
\end{minipage}
}

\newcommand{\ugettikzxy}[3]{%
  \tikz@scan@one@point\pgfutil@firstofone#1\relax
  \edef#2{\the\pgf@x}%
  \edef#3{\the\pgf@y}%
  }

\def\g#1#2{g(#1,#2)}
\def\G#1#2{G(#1,#2)}
\def\au{\underline{\alpha}}
\def\ao{\overline{\alpha}}
\usepackage[all]{xy}
\numberwithin{equation}{section} \numberwithin{proposition}{section}
\numberwithin{lemma}{section}

\begin{document}
\title{Modernizing Archimedes' Construction of $\pi$}
\author{David Weisbart}\address{Department of Mathematics\\University of
California, Riverside}\email{weisbart@math.ucr.edu}

\maketitle 

\begin{flushright}
In memory of my mentor and dear friend, Professor V.S.Varadarajan.
\end{flushright}

\begin{abstract}
In his famous work, ``Measurement of a Circle,'' Archimedes described a procedure for measuring both the circumference of a circle and the area it bounds.  Implicit in his work is the idea that his procedure defines these quantities.  Modern approaches for defining $\pi$ eschew his method and instead use arguments that are easier to justify, but they involve ideas that are not elementary.  This paper makes Archimedes' measurement procedure rigorous from a modern perspective.  In so doing, it brings a rigorous and geometric treatment of the differential properties of the trigonometric functions into the purview of an introductory calculus course.
\end{abstract}

\setcounter{tocdepth}{1}
\tableofcontents

\section{Introduction}

Archimedes' estimate of the value of $\pi$ in ``Measurement of a Circle'' is one of his greatest achievements \cite{Heath}.  His approach anticipated foundational ideas of modern analysis as formulated in the 19th and early 20th centuries, making this work one of the first major analytical achievements.  Archimedes implicitly defined the circumference of a unit circle axiomatically as a number always greater than the perimeters of approximating inscribed polygons and always less than the perimeters of approximating circumscribed polygons, where the approximating polygons are regular refinements of a regular hexagon.  By calculating the perimeters of a regular inscribed and circumscribed polygon of 96 sides, he arrived at the famous estimate \begin{flalign*}&\dsp \frac{223}{71} < \pi < \frac{22}{7}.&\end{flalign*}  He also found upper and lower bounds for the area of a disk respectively given by the areas of regular refinements of circumscribed and inscribed squares. He showed that the area of the disk should be the area of a right triangle with one leg having length equal to the circumference of the unit circle and the other having the length of the radius.

 The fact that the approximation procedure for the circumference appears to depend both on the method of refinement and on the type of the polygons used in the first stage of the approximation presents a difficulty from a modern point of view. The circumference should be intrinsic to the circle and independent of any specified approximation procedure.  It is important to modernize Archimedes' construction in a way that is maximally accessible to a contemporary of Archimedes for both aesthetic and practical reasons.  A practical consequence is that such a modernization cleanly and efficiently resolves the principle difficulty that arises in developing the infinitesimal theory of the trigonometric functions, the calculation \begin{flalign}\label{limsinxoverx} &\dsp \lim_{x\to 0}\frac{\sin(x)}{x} = 1.&\end{flalign} Unger notes in \cite{Unger} some difficulties in the approach common in differential calculus textbooks.  Richman also points out in \cite{Rich93} some of these difficulties and directs the reader to the approach of Apostol \cite{Apostol}, that avoids common circular arguments by using twice the area of a sector rather than arclength as the argument of the trigonometric functions.  This is a natural point of view, especially when considering the direct analogy with the hyperbolic trigonometric functions.  Although Apostol's approach is compelling, it remains desirable to connect this nonstandard way of defining the trigonometric functions with the standard approach.  Moreover, Apostol defines the area of a sector rather than deriving it from a limiting procedure using inscribed and circumscribed polygons.  So there is still some work to be done to connect the viewpoints.

There are several approaches to computing \eqref{limsinxoverx} in the literature.  Authors commonly use an analytical approach to solve the problem by defining the trigonometric functions as power series or as solutions to a system of differential equations.  Some authors begin by defining the inverse of the tangent function as an integral.  Unfortunately, these approaches lack geometric motivation.  In \cite{Vietoris}, Vietoris used the sum of angles formula for the trigonometric functions to prove \eqref{limsinxoverx} directly.  The notion of an arbitrary fraction of a circle requires a group structure on the circle and defining such a structure is certainly natural and unavoidable. The argument of the sine and cosine functions will be a multiple of the fraction of a circle that a particular arc represents.  However, taking the multiple to be the length of an arc requires a definition of the length of an arc.  Defining and computing the length of the arc is the primary difficulty.  

If one follows Vietoris, one must define~$\pi$ in some way independent of it being the area of the unit circle, or half its circumference.  If one uses the approach of Zeisel in \cite{Zeisel} that follows Vietoris'---at least philosophically---then one can dispense with the alternate definition of~$\pi$.  Taking ~$\pi$ as the limit of the area given by regular circumscribed $n$-gons and taking the multiple of the fraction in the argument to be~$2\pi$ will imply the limit \eqref{limsinxoverx}.  However, without some further work, even this elegant approach of Zeisel does not establish the meaning of the area and circumference of the unit circle as well as their relationship.  In fact, one must still show that the limit of the areas he discusses actually exists.

Using only basic euclidean geometry, we prove that for any arc $\mathscr A$ that is less than half of a circle, if $n$ is greater than $m$ and if $\ell_n$ and $\ell_m$ are the lengths of chords that respectively subtend arcs that are an $n^{\rm th}$ and an $m^{\rm th}$ of $\mathscr A$, then \begin{flalign*} &\dsp m\ell_n >  n\ell_m.&\end{flalign*}  We prove an analogous but reverse inequality for lengths corresponding to edges of circumscribed polygons. These inequalities imply that the circle is a rectifiable curve.  They furthermore give the classical relationship between the circumference of a circle and the area that it bounds in a way that is both rigorous and accessible to freshman calculus students.  Using these inequalities, we show that $2\pi$ is the limiting value of {\it any} sequence of perimeters of approximating polygons.  This approach is a simple modernization of the approach of Archimedes and makes rigorous the usual geometric arguments in calculus textbooks that prove \eqref{limsinxoverx}.  It is also elementary enough to be accessible, at least in principle, to a freshman calculus student.

\subsection*{Acknowledgements}

I thank Professor V.S.Varadarajan for the insights and suggestions he gave me during our many discussions involving this paper.  From him I learned to love, among many things, the history of our subject, and so I dedicate this paper to his memory.  I thank Dr.\;Alexander Henderson for editing this paper and for his helpful comments.

\section{Approximation by Regular $2^mn$-gons}

Denote by $\C$ the unit circle centered at $O$.  Unless otherwise specified, $n$ will be in $\mathds N_{\ge 3}$, the set of natural numbers greater than or equal to three, and $m$ will be in $\mathds N_0$, the natural numbers with zero.  All sequences henceforth indexed by $m$ will be indexed over the set $\mathds N_0$.

\subsection{Regular Inscribed and Circumscribed Polygons}

\begin{definition}
An \emph{inscribed polygon} of $\C$ is a simple polygon all of whose vertices lie on $\C$ and \emph{circumscribed polygon} of $\C$ is a simple polygon all of whose edges intersect $\C$ tangentially.  An inscribed polygon of $\C$ is \emph{regular} if all of its edges are of equal length.  A circumscribed polygon of $\C$ is \emph{regular} if each of its edges intersects $\C$ tangentially at its midpoint.
\end{definition}

While circumscribed polygons are often assumed by definition to be convex, this assumption is redundant.  Furthermore, our definition of regularity implies that all edges of a regular circumscribed polygon are congruent.  Denote respectively by $\g{m}{n}$ and $\G{m}{n}$ the regular inscribed and circumscribed polygons with $2^mn$ edges. Up to congruency, there is only one such inscribed and one such circumscribed polygon for each choice of $m$ and $n$.  Denote respectively by $\ell_n(m)$ and $L_n(m)$ the edge length of $\g{m}{n}$ and $\G{m}{n}$.  The respective perimeters of $\g{m}{n}$ and $\G{m}{n}$ are $p_n(m)$ and $P_n(m)$, where \begin{flalign*}&\dsp p_n(m) = 2^mn\ell_n(m) \quad {\rm and}\quad  P_n(m) = 2^mnL_n(m).&\end{flalign*}  Let $P$ and $Q$ be adjacent vertices of $\g{m}{n}$ and let $A$ and $B$ be adjacent vertices of $\G{m}{n}$.  Denote respectively by $\au_n(m)$ and $\ao_n(m)$ the areas of the triangles $\triangle POQ$ and $\triangle AOB$. The respective areas of $\g{m}{n}$ and $\G{m}{n}$ are $a_n(m)$ and $A_n(m)$, where \begin{flalign*}&\dsp a_n(m) = 2^mn\au_n(m)\quad {\rm and}\quad A_n(m) = 2^mn\ao_n(m).&\end{flalign*}

The triangle inequality and the additivity of area together imply the following proposition.

\begin{proposition}\label{Proposition:littlepanda}
For fixed $n$, the sequences $(p_n(m))$ and  $(a_n(m))$ are both increasing and the sequences $(P_n(m))$ and $(A_n(m))$ are both decreasing.
\end{proposition}

The parallel postulate implies the following useful lemma.  Refer to Figure~\ref{e} to clarify the statement of the lemma.

\medskip

\dwpice

\medskip

\begin{lemma}\label{Lemma:LemmaMain}
Suppose that $C$ and $B$ are points on $\C$ so that the counterclockwise oriented arc from $C$ to $B$ is less than half of $\C$. If $L_B$ and $L_C$ are lines tangent to $\C$ respectively at $B$ and $C$, then $L_B$ and $L_C$ intersect at a point $A$ and the right triangles $\triangle OAB$ and $\triangle OCA$ are congruent.  If $F$ is the point of intersection of $\overline{OA}$ with $\overline{BC}$, then $\angle AFB$ is a right angle.  Let $M$ be the point at which $\overline{OA}$ intersects $\C$ and $L_M$ be the line tangent to $\C$ at $M$.  Denote by $D$ the point at which the line $L_M$ intersects $\overline{BA}$ and denote by $E$ the point at which $L_M$ intersects $\overline{CA}$.  The angle $\angle AMD$ is congruent to $\angle EMA$ and both are right angles.
\end{lemma}

\subsection{Area and Perimeter Bounds}

While Proposition \ref{Proposition:littlepanda} guarantees for fixed $n$ the strict monotonicity of $(p_n(m))$, $(a_n(m))$, $(P_n(m))$ and $(A_n(m))$, it does not guarantee that the sequences are bounded. The following proposition provides the desired bounds.

\begin{proposition}\label{prop:seqbounds}
For each fixed $n$, \begin{flalign*}&\dsp p_n(m) < P_n(m) \quad {\rm and}\quad  a_n(m) < A_n(m).&\end{flalign*}
\end{proposition}

Since $(p_n(m))$ and $(a_n(m))$ are increasing and bounded above by $P_n(0)$ and $A_n(0)$, and since $(P_n(m))$ and $(A_n(m))$ are decreasing and bounded below by $p_n(0)$ and $a_n(0)$, all four sequences are convergent, implying Proposition~\ref{pPaAn}.

\begin{proposition}\label{pPaAn}
For each fixed $n$, there are real numbers $p_n$, $P_n$, $a_n$ and $A_n$ such that \begin{flalign*}&\dsp p_n(m)\to p_n,\quad P_n(m) \to P_n,\quad a_n(m)\to a_n,\quad {\rm and} \quad A_n(m) \to A_n.&\end{flalign*}
\end{proposition}

\subsection{Convergence of the Approximations}

Proposition~\ref{pPaAn} gives for each fixed $n$ respective limiting values for the areas and perimeters of regular refinements of regular inscribed and circumscribed polygons with $n$ edges.  It does not prove the equality of the respective limits.

\begin{theorem}[Heron's Theorem] If $T$ is a triangle with side lengths $a,b,c$ and $A(T)$ is the area of $T$, then \begin{flalign*}&\dsp A(T) = \sqrt{s(s-a)(s-b)(s-c)} \quad {\rm where} \quad s = \frac{a+b+c}{2}.&\end{flalign*}
\end{theorem}

For any points $A$ and $B$ in the plane, denote by $\ell(\overline{AB})$ the length of the line segment $\overline{AB}$.  Denote by $h_n(m)$ the distance from a vertex of $G(m,n)$ to $\C$.

\begin{proposition}\label{Proposition:PropositionTwoSix}
Given $A_n$ and $a_n$ above, \begin{flalign}\label{4eqs}&\dsp (i)\; a_n = \frac{1}{2}p_n,\quad (ii)\; A_n = \frac{1}{2}P_n,\quad (iii)\; P_n = p_n,\quad  \text{and}\quad(iv)\quad A_n = a_n.&\end{flalign}
\end{proposition}

\begin{proof}
Since the sequences $(p_n(m))$ and $(P_n(m))$ are both convergent and are respectively equal to $(2^mn\ell_n(m))$ and $(2^mnL_n(m))$, both $\ell_n(m)$ and $L_n(m)$ tend to zero as $m$ tends to infinity.  Heron's theorem implies that \begin{flalign*}&\dsp{\underline{\alpha}}_n(m) = \sqrt{\left(1 + \tfrac{1}{2}\ell_n(m)\right)\left(1-\tfrac{1}{2}\ell_n(m)\right)\left(\tfrac{1}{2}\ell_n(m)\right)^2} = \tfrac{1}{2}\ell_n(m)\sqrt{1-\tfrac{1}{4}\ell_n(m)^2},&\end{flalign*}  and so \begin{flalign}\label{LittleAreatoPerimeter}&\dsp a_n(m) = 2^mn\tfrac{1}{2}\ell_n(m)\sqrt{1-\tfrac{1}{4}\ell_n(m)^2}= \tfrac{1}{2}p_n(m)\sqrt{1-\tfrac{1}{4}\ell_n(m)^2}\to \tfrac{1}{2}p_n.&\end{flalign} Since $\C$ is a unit circle, $\overline{\alpha}_n(m)$ is equal to $\tfrac{1}{2}L_n(m)$ and so \begin{flalign}\label{BigAreatoPerimeter}&\dsp A_n(m) = 2^mn\overline{\alpha}_n(m) = 2^{m-1}nL_n(m) = \tfrac{1}{2}P_n(m)\to \tfrac{1}{2}P_n.&\end{flalign}

Suppose that $G$ is a vertex of $\G{m}{n}$, that $A$ and $H$ are points on $\C$ where the edges of $\G{m}{n}$ with endpoint $G$ are tangent to $\C$, and that $\triangle GAH$ is counterclockwise oriented (Figure~\ref{f}).  Line segments $\overline{AG}$ and $\overline{GH}$ are half edges of $\G{m}{n}$, so $\ell(\overline{AG})$ is equal to $\tfrac{1}{2}L_n(m)$.  Take $N$ to be a point of $\C$ so that $\overline{AN}$ and $\overline{NH}$ are edges of $\g{m+1}{n}$.  The intersection, $C$, of lines tangent to $\C$ at $A$ and $N$ is a vertex of $\G{m+1}{n}$.  Take $P$ to be the intersection of $\overline{AH}$ with $\overline{OP}$ and $M$ to be the intersection of $\overline{AN}$ with $\overline{OC}$.  Let $I$ be the point of $\C$ that intersects $\overline{OC}$.  Lemma~\ref{Lemma:LemmaMain} implies that $\angle GPA$ and $\angle CMA$ are right angles.   Denote by $B$ intersection of line tangent to $\C$ at $A$ with line tangent to $\C$ at $I$.  Denote by $J$ the intersection of the line tangent to $\C$ at $I$ with line tangent to $\C$ at $N$.  The points  $B$ and $J$ are neighboring vertices of $\G{m+2}{n}$.  Angle $\angle GNA$ is obtuse because $\angle GNC$ is right, hence \begin{flalign*}&\dsp\tfrac{1}{2}L_n(m) > \ell_n(m+1).&\end{flalign*}  
 
Extend $\overline {BJ}$ to meet $\overline{OG}$ at a point $L$.  Lemma \ref{Lemma:LemmaMain} implies that $\angle OMN$ and $\angle OIJ$ are both right angles and so $\overline{BJ}$ and $\overline{AN}$ are parallel.  The point $L$ therefore lies between $N$ and $G$. The line segment $\overline{NG}$ has length $h_n(m)$ and $\ell(\overline{IC})$ is equal to $h_n(m+1)$.

\medskip
 
\dwpicf

\medskip

Line segments $\overline{BI}$, $\overline{IJ}$, and $\overline{JN}$ are congruent as half edges of $\G{m+2}{n}$. The hypotenuse of the right triangle $\triangle JNL$ is $\overline{JL}$, so $\ell(\overline{JL})$ is greater than $\ell(\overline{BI})$. Let $K$ be a point on $\overline{JL}$ such that $\overline{JK}$ and $\overline{BI}$ are congruent.  Furthermore, let $E$ and $F$ be points lying on $\overline{AG}$ such that $\overline{KE}$, $\overline{LF}$, and $\overline{IC}$ are parallel. The similarity of $\triangle BIC$, $\triangle BKE$, and $\triangle BLF$ and the equality of $\ell(\overline{BK})$ and $3\ell(\overline{BI})$ together imply that $\ell(\overline{KE})$ is equal to $3\ell(\overline{IC})$.  Since $\angle OAC$ is right, $\angle OCG$ is obtuse and so $\angle LFG$ is as well, implying that $\ell(\overline{LG})$ is greater than $\ell(\overline{LF})$. Since $\ell(\overline{BL})$ is greater than $\ell(\overline{BK})$, \begin{flalign*}&\dsp h_n(m) = \ell(\overline{NG}) > \ell(\overline{LG}) > \ell(\overline{LF}) > \ell(\overline{KE}) = 3\ell(\overline{IC}) = 3h_n(m+1).&\end{flalign*}  Since $m$ was arbitrary, \begin{flalign*}&\dsp h_n(m) < \left(\tfrac{1}{3}\right)^{m}h_n(0).&\end{flalign*} 

The triangle inequality implies that $h_n(m) + \ell_n(m+1)$ is greater than $\tfrac{1}{2}L_n(m)$, hence \begin{flalign*}&\dsp 0 < \tfrac{1}{2}L_n(m) - \ell_n(m+1) < h_n(m) < \left(\tfrac{1}{3}\right)^mh_n(0),&\end{flalign*} and so \begin{flalign}\label{new:eq:a}&\dsp 0 < 2^mnL_n(m) - 2^{m+1}n\ell_n(m+1) = P_n(m) - p_n(m+1) < 2n\left(\tfrac{2}{3}\right)^mh_n(0).&\end{flalign}  The estimate \eqref{new:eq:a} implies that \begin{flalign*}&\dsp 0 < P_n(m) - p_n(m)  = (P_n(m) - p_n(m+1)) + (p_n(m+1) - p_n(m))&\\&\dsp \phantom{0} <  2n\left(\tfrac{2}{3}\right)^mh_n(0) + (p_n(m+1) - p_n(m)) \to 0.&\end{flalign*}  The difference $P_n(m)-p_n(m)$ tends to zero as $n$ tends to infinity, hence $P_n$ is equal to $p_n$.  Of course, (\ref{Proposition:PropositionTwoSix}.iv) follows immediately from \eqref{BigAreatoPerimeter} and \eqref{LittleAreatoPerimeter}.
\end{proof}

\section{Edge Length Comparison Theorems}

Comparing the edge lengths of inscribed and circumscribed segments corresponding to different regular subdivisions of an arc is the key to proving that the definition of $\pi$ is independent of any approximation scheme.  This section will present such a comparison. 

\begin{definition}
A \emph{\lpar counter\;\rpar clockwise oriented partition $P$ of an arc $\A$} is a finite sequence of points of $\A$ that is (counter)clockwise ordered, and the first and last points of $P$ are the endpoints of $\A$.  Such a finite sequence is \emph{regular} if all adjacent points of $P$ are equidistant. 
\end{definition}

Temporarily ignore the previous restrictions on the natural numbers $m$ and $n$.

\medskip

\dwpicg

\bigskip

Suppose that $m$ and $n$ are natural numbers with $n$ greater than $m$.  Let $\mathscr A$ be an arc of $\C$ that is less than half of $\C$ and let $(P_1, \dots, P_{n+1})$ be a regular clockwise oriented partition of $\mathscr A$. Let $B$ be the point of intersection of the lines tangent to $\C$ at $P_1$ and $P_{n+1}$ and let $A$ be the point of intersection of the lines tangent to $\C$ at $P_1$ and $P_{M+1}$.  Denote by $\ell_m$, $\ell_n$, $L_m$, and $L_n$ the respective lengths of $\overline{P_1P_{m+1}}$, $\overline{P_1P_{n+1}}$, $\overline{P_1B}$ and $\overline{P_1A}$.  We will show that \begin{flalign*}&\dsp n\ell_m > m\ell_n \quad {\rm and}\quad nL_m < mL_n.&\end{flalign*}

\subsection{General Symmetry Considerations For Polygonal Segments}

Assume that $A$, $B$, $C$, and $D$ belong to a collection of points that form a regular partition of $\mathscr A$ (Figure~\ref{h}), that $A$ and $C$ are adjacent, that $B$ and $D$ are adjacent, and that $C$ and $D$ lie on the same side of the line $\overline{AB}$ and opposite the side where $O$ lies.  Suppose furthermore that the quadruple $(A,B,D,C)$ forms a counterclockwise oriented partition of $\A$.

\medskip

\noindent\begin{minipage}{.49\linewidth}
\dwpich
\end{minipage}
\begin{minipage}{.49\linewidth}
\dwpichh
\end{minipage}

\bigskip

Standard arguments using the SSS theorem imply Proposition~\ref{Proposition:IsocTrap}, whose proof is left as an exercise.
 
\begin{proposition}\label{Proposition:IsocTrap}
The quadrilateral $\square ABDC$ \lpar Figure~\ref{h} \rpar\,is an isosceles trapezoid and $\overline{AB}$ and $\overline{CD}$ are parallel. Furthermore, in the degenerate case \lpar Figure~\ref{hh} \rpar\,when $C$ is equal to $D$, angles $\angle BAC$ and $\angle ABC$ are congruent and $\overline{CO}$ bisects $\overline{AB}$.
\end{proposition}

\medskip
 
\dwpici

\bigskip

Suppose that $\mathscr A$ is an arc of $\C$ that is less than half of $\C$ and suppose that $(P_1, \dots, P_n)$ is a regular clockwise orientated partition with $n$ greater than 3 (Figure~\ref{i}). For each natural number $i$ in $[1, n]$, let $L_i$ be the line tangent to $\A$ and intersecting $P_i$. Since $\mathscr A$ is less than half of a circle, Lemma \ref{Lemma:LemmaMain} implies that the lines $L_i$ and $L_j$ intersect for each $i$ not equal to $j$; call this point of intersection $P_{i,j}$. Notice that  $P_{i,j}$ is equal to $P_{j,i}$. The key theorem of the next section is to prove that \begin{flalign*}&\dsp j > i \quad \text{implies that}\quad  \ell(\overline{P_{1,i}P_{1,i+1}}) < \ell(\overline{P_{1,j}P_{1,j+1}}).&\end{flalign*}  

Refer to Figure \ref{j}.  Suppose that $(A, B, C, D)$ is a clockwise ordered partition of an arc $\A$ of $\C$ and that $\overline{AB}$ has the same length as $\overline{CD}$. Let $L_A$, $L_B$, $L_C$ and $L_D$ be lines tangent to $\mathscr A$ that intersect $A$, $B$, $C$, and $D$ respectively.  Since $\mathscr A$ is less than half of a circle,  Lemma \ref{Lemma:LemmaMain} implies that $L_A$ and $L_D$ intersect at a point $P_{AD}$, $L_A$ and $L_C$ intersect at a point $P_{AC}$, $L_B$ and $L_D$ intersect at a point $P_{BD}$, and $L_B$ and $L_C$ intersect at a point $P_{BC}$. The SAS and SSS theorems along with Proposition~\ref{Proposition:IsocTrap} and Lemma~\ref{Lemma:LemmaMain} imply the following proposition, whose proof is left to the reader as a standard exercise.

\medskip

\dwpicj

\bigskip

\begin{proposition}\label{Proposition:Prop3point2}
The line segment $\overline{P_{AD}P_{BC}}$ can be extended to a ray, $L$, that meets $O$ and bisects and is perpendicular to $\overline{P_{AC}P_{BD}}$, $\overline{AD}$, and $\overline{BC}$. Furthermore, $L$ bisects $\angle P_{AC}P_{AD}P_{BD}$. Finally, triangles $\triangle P_{AD}P_{AC}P_{BC}$ and $\triangle P_{AD}P_{BD}P_{BC}$ are congruent, as are $\triangle P_{AD}P_{BC}P_{AB}$ and $\triangle P_{AD}P_{BC}P_{CD}$.
\end{proposition}

Denote by $m(\angle ABC)$ the degree measure of $\angle ABC$.  Refer to Figure~\ref{i} for Proposition~\ref{Proposition:Anglemeasure}.  

\begin{proposition}\label{Proposition:Anglemeasure}
Suppose that $n$ is in $\mathds N_{\ge 3}$ and $(P_1, \dots, P_n)$ is a regular clockwise oriented partition of an arc $\mathscr A$ that is less than half of $\C$.  If $m(\angle P_iOP_{i+1})$ is equal to $\theta^\circ$, then \begin{flalign*}&\dsp (1)\quad m(\angle P_{1,n-1}P_{1,n}P_{n,2}) = (180 - n\theta)^\circ;&\\&\dsp (2)\quad m(\angle P_{2,n-1}P_{1,n-1}P_{1,n}) = m(\angle P_{n,1}P_{n,2}P_{2,n-1}) = (n-1)\theta^\circ; &\\&\dsp (3)\quad m(\angle P_{n,2}P_{2,n-1}P_{1,n-1}) = 180^\circ - (n-2)\theta^\circ.&\end{flalign*}
\end{proposition}

\begin{proof}
Both $\angle P_{1,n}P_nO$ and $\angle OP_1P_{1,n}$ are right angles, implying that \begin{flalign*}&\dsp m(\angle P_nOP_1) + m(\angle P_1P_{1,n}P_n) = 180^\circ.&\end{flalign*}  Since $m(\angle P_nOP_1)$ is equal to $n\theta^\circ$, the fact that the sum of the interior angles of the quadrilateral formed by $P_1$, $P_n$, $O$ and $P_{1,n}$ is $360^\circ$ implies (1).  Since $P_1$, $P_{1,n-1}$ and $P_{1,n}$ lie on the same line, $\angle P_1P_{1,n-1}P_{2,n-1}$ and $\angle P_{2,n-1}P_{1,n-1}P_{1,n}$ are supplementary.  By (1), \begin{flalign*}&\dsp m(\angle P_{1,1}P_{1,n-1}P_{2,n-1})= (180 - (n-1)\theta)^\circ&\end{flalign*} and so \begin{flalign*}&\dsp m(\angle P_{1,n}P_{1,n-1}P_{2,n-1}) = (n-1)\theta^\circ.&\end{flalign*} Of course, an analogous argument proves the same result for $\angle P_{2,n-1}P_{n,2}P_{n}$, implying (2), although another argument is not necessary in light of Proposition~\ref{Proposition:Prop3point2}.  Finally, since the sum of angles of a quadrilateral is $360^\circ$, (3) follows from (1) and (2).
\end{proof}

\subsection{Length Estimates For Inscribed Polygonal Segments} Let $\mathscr A$ again be an arc of a unit circle that is less than half of the circle.  Suppose that  $(P_1, \dots, P_{n+1})$ is a regular clockwise oriented partition of $\mathscr A$ (Figures \ref{k} and \ref{l}). For each $P_i$ where $i$ is a natural number in $[2, n]$, there is a line perpendicular to $\overline{P_1P_{n+1}}$ that intersects $P_i$ at a point $q_i$ on $\overline{P_1P_{n+1}}$.

Figure \ref{k} shows an example where $n$ is even and Figure~\ref{l} shows the qualitative difference in an odd example. Set $q_1$ to be equal to $P_1$ and let $q_{n+1}$ be equal to $P_{n+1}$.  For each $i$ between $1$ and $n+1$, let $q_i$ be the point on $\overline{P_1P_{n+1}}$ so that $\overline{P_iq_i}$ is perpendicular to $\overline{P_1P_{n+1}}$.

\medskip

\noindent\begin{minipage}{.49\linewidth}
\dwpick
\end{minipage}
\begin{minipage}{.49\linewidth}
\dwpicl
\end{minipage}

\medskip

\begin{proposition}\label{Proposition:AverageLengthProp}
If $i$ and $j$ are natural numbers in $\left[1, \frac{n+1}{2}\right]$ with $i$ less than $j$, then
\begin{flalign*}&\dsp \ell(\overline{q_iq_{i+1}}) < \ell(\overline{q_jq_{j+1}}) \quad {\rm and}\quad \ell(\overline{q_iq_{i+1}}) = \ell(\overline{q_{n+1 -i}q_{n+2 - i}}).&\end{flalign*}
\end{proposition}

\begin{proof}
Proposition \ref{Proposition:IsocTrap} implies that each of the line segments $\overline{P_iP_{n+2-i}}$ is parallel to $\overline{P_1P_{n+1}}$.  Since $\overline{P_iq_i}$ and $\overline{P_iP_{n+2-i}}$ are perpendicular for $i$ in $\left[1, \frac{n}{2}\right]$, \begin{flalign*}&\dsp \ell(\overline{P_iP_{n+2-i}}) = \ell(\overline{q_iq_{n+2-i}}).&\end{flalign*}  Let $X_1$ be equal to $q_2$ and, for each $i$ in $\left[2, \frac{n}{2}\right]$, let $X_{i}$ be the point of intersection of $\overline{P_iP_{n+2-i}}$ and $\overline{P_{i+1}q_{i+1}}$. Let $X_{n-1}$ be equal to $q_{n}$ and, for each $i$ in $\left[\frac{n}{2} +1, n\right]$, let $X_{i}$ be the point of intersection of $\overline{P_{i+2}P_{n-i}}$ and $\overline{P_{i+1}q_{i+1}}$.  For each $i$ in $\left[1, \frac{n}{2}\right]$, Proposition~\ref{Proposition:IsocTrap} implies that  $\ell(\overline{P_iX_{i}})$ is equal to $\ell(\overline{X_{n-i}P_{n+2-i}})$, which is equal to $\ell(\overline{q_iq_{i+1}})$ and $\ell(\overline{q_{n+1-i}q_{n+2-i}})$.

Let $i$ and $j$ be in $\left[1, \frac{n}{2}\right]$ with $i$ less than $j$.  To show that $\ell(P_iX_{i})$ is less than $\ell(P_jX_{j})$, it suffices to show that if $i$ and $i+1$ are in $\left[1, \frac{n}{2}\right]$, then $\ell(P_iX_{i})$ is less than $\ell(P_{i+1}X_{i+1})$.   Extend the line segment $\overline{P_{i}P_{i+1}}$ to a ray, $R_1$, originating at $P_{i}$.  Extend the line segment $\overline{q_{i+2}P_{i+2}}$ to a ray, $R_2$, that meets $R_1$ at a point $A$.  The ray $R_2$ passes through the point $X_{i+1}$ and $l(\overline{X_{i+1}A})$ is greater than $l(\overline{X_{i+1}P_{i+2}})$, while angles $\angle X_{i+1}P_{i+1}A$ and $\angle X_{i}P_{i}P_{i+1}$ are congruent, implying that $\angle X_{i}P_{i}P_{i+1}$ is greater than $\angle X_{i+1}P_{i+1}P_{i+2}$.  Since $\triangle X_{i}P_{i}P_{i+1}$ and $\triangle X_{i+1}P_{i+1}P_{i+2}$ are both right triangles with hypotenuses that have the same length, $\angle P_{i}P_{i+1}X_{i}$ is greater than $\angle P_{i+1}P_{i+2}X_{i+1}$, implying that $\ell(\overline{P_{i}X_{i}})$ is less than $\ell(\overline{P_{i+1}X_{i+1}})$.  Since $i$ and $j$ are integers, the statement of the proposition is different when $n$ is odd and when $n$ is even.

Note that if $n$ is odd, then the above arguments show that $\overline{q_{\frac{n+1}{2}}q_{\frac{n+1}{2}+1}}$ is the longest of all of the segments $\overline{q_iq_{+1}}$.
\end{proof}

Since the lengths of the line segments $\overline{q_iq_{i+1}}$ are strictly increasing up to the midpoint of $\mathscr A$ for even $n$ and the middle segment for odd $n$, the following corollary is immediate.

\begin{corollary}
Let $k$ be the largest integer less than or equal to $\frac{n}{2}$.  If $n$ is greater than $2$ and $s$ is a natural number in $(1, k]$,  then 
\begin{flalign*}&\dsp\ell(\overline{q_1q_{s+1}}) < \frac{\ell(\overline{P_1P_{n+1}})}{n}s.&\end{flalign*}
\end{corollary}

\begin{remark}
The above corollary implies that the average length of a segment of $\overline{q_1q_{k+1}}$ is less than the average length of a segment of $\overline{q_1q_{n_1}}$ if the arc between $P_1$ and $P_{k+1}$ is less than half of $\mathscr A$.
\end{remark}

From this corollary follows the following key theorem.

\begin{theorem}\label{Theorem:TheoremTwo}
Suppose that $\A$ is an arc that is less than half of a circle and that $(P_1, \dots, P_{n+1})$ is a regular clockwise oriented partition of $\A$.  Denote by $\ell_n$ the length of the line segment $\overline{P_1P_{n+1}}$, and by  $\ell_m$ the length of $\overline{P_1P_{m+1}}$.  If $m$ is less than $n$, then  \begin{flalign*}&\dsp n\ell_m > m\ell_n.&\end{flalign*}
\end{theorem}

\begin{proof}
Assume first that $2m$ is greater than $n$ and refer to Figure~\ref{m}.  Let $Q$ be the point of intersection of the circle of radius $\ell_m$ centered at $P_1$ with the line segment $\overline{P_1P_{n+1}}$.  The length of $\overline{P_1Q}$ is equal to $\ell_m$. If $s$ the length of the line segment $\overline{QP_{n+1}}$, then  \begin{flalign*}&\dsp s = \ell_n-\ell_m.&\end{flalign*} There are $n$ edges with vertices on the arc between $P_1$ and $P_{n+1}$.  There are $m+1$ vertices of the regular subdivision between $P_1$ and $P_{m+1}$ on $\A$. Since $2m$ is greater than $n$, $P_{m+1}$ lies clockwise from the intersection, $M$, of a line bisecting $\overline{P_1P_{n+1}}$ and the circle. For the same reason, the point $P_{n - m}$ lies counterclockwise from $M$. 

Proposition \ref{Proposition:AverageLengthProp} implies that the average length of the segments $\overline{q_1q_2}$, $\overline{q_2q_3}$, \dots, $\overline{q_{n-m}q_{n-m}}$ is the same as the average lengths of the segments $\overline{q_{n+1}q_n}$, $\overline{q_nq_{n-1}}, \dots, \overline{q_{2+m}q_{1+m}}$. However, for any edge $\overline{P_iP_{i+1}}$ between the vertices $P_{n+1-m}$ and $P_{m+1}$, the projection of the edge, $\overline{q_iq_j}$ on the line $\overline{P_1P_{n+1}}$, is longer than the longest of the segments $\overline{q_jq_{j+1}}$ with $j$ in $[1, n-m)\cup(m+1, n]$.  Denote by $d$ the length of the line segment $\overline{q_{m+1}q_{n+1}}$ to obtain the inequality \begin{flalign*}&\dsp d < \dfrac{n-m}{n}\ell_n.&\end{flalign*} Since $\ell_m$ is the length of the hypotenuse of a triangle with a leg of length $\ell_n - d$ and $\ell_m$ is equal to $\ell_n-s$, the length $s$ is less than $d$ and so \begin{flalign*}&\dsp s < \dfrac{n-m}{n}\ell_n.&\end{flalign*} Since, $\ell_m + s$ is equal to $\ell_n$, \begin{flalign*}&\dsp \ell_m + \dfrac{n-m}{n}\ell_n > \ell_n, \quad \text{and so} \quad n\ell_m > m\ell_m.&\end{flalign*}

\bigskip

\dwpicm

\bigskip

Suppose that $2m$ is equal to $n$.  Since the sum of the lengths of two sides of a non-degenerate triangle are greater than the length of the third, $2\ell_m$ is greater than $\ell_{2m}$, and so \begin{flalign}\label{ins:edge:comp}&\dsp  \quad n\ell_m = 2m\ell_m > m\ell_{2m} = m\ell_n.&\end{flalign}

In the case when $2m$ is less than $n$, there is a natural number $k$ with \begin{flalign*}&\dsp 2^km < n \leq 2^{k+1}m&\end{flalign*}
which together with \eqref{ins:edge:comp} implies that 
\begin{flalign*}&\dsp n\ell_{2^km} > 2^km\ell_n.&\end{flalign*} The points $P_1$ and $P_{2^{k-1}m +1}$ have $2^{k-1}m$ segments between them as do the points $P_{2^{k-1}m +1}$ and $P_{2^{k}m+1}$. Thus,  \begin{flalign*}&\dsp \ell(\overline{P_1P_{2^{k-1}m +1}}) = \ell(\overline{P_{2^{k-1}m +1}P_{2^{k}m+1}}) = \ell_{2^{k-1}m}.&\end{flalign*}  Therefore, $2\ell_{2^{k-1}m}$ is equal to the sum of the lengths of two sides of $\triangle P_1P_{2^{k-1}m+1}P_{2^km+1}$ and so is greater than $\ell_{2^km}$. Therefore,  \begin{flalign*}&\dsp 2n\ell_{2^{k-1}m} > n\ell_{2^km} > 2^km\ell_n\quad \text{and so}\quad n\ell_{2^{k-1}m} > 2^{k-1}m\ell_n > m\ell_n.&\end{flalign*} Applying the above argument $k-1$ more times yields the inequality \begin{flalign*}&\dsp n\ell_m > m\ell_n.&\end{flalign*}
\end{proof}

\subsection{Length Estimates For Circumscribed Polygonal Segments}

Use the notation of Proposition \ref{Proposition:Anglemeasure} for the statement and proof of the proposition below.
\begin{proposition}
If $k$ and $l$ are natural numbers, then
\begin{flalign*}&\dsp 1\leq k < l \leq n-1 \quad\text{implies that} \quad \ell(\overline{P_{1,k}P_{1,k+1}}) <
\ell(\overline{P_{1,l}P_{1,l+1}}).&\end{flalign*}
\end{proposition}

\bigskip

\begin{minipage}{.49\linewidth}
\dwpicn
\end{minipage}
\begin{minipage}{.49\linewidth}
\dwpico
\end{minipage}

\bigskip

\begin{proof}
Suppose that $i$ is a natural number in $[3,n]$. Proposition \ref{Proposition:Prop3point2} implies that the line $\overline{P_{1,i-1}P_{2,i-2}}$ is a perpendicular bisector of $\overline{P_{1,i-2}P_{2,i-1}}$. Denote by $A$ this intersection. To prove the proposition, it suffices to show that $\angle P_{1,i}P_{2,i-1}P_{1,i-2}$ is greater than a right angle. In this case, there is a line that intersects $P_{2,i-1}$ and is perpendicular to $\overline{P_{1,i-2}P_{2,i-1}}$ that intersects $\overline{P_{1,i-2}P_{1,i}}$ at a point $B$.  The triangles $\triangle P_{1,i-2}AP_{1,i-1}$ and $\triangle P_{1,i-2}P_{2,i-1}B$ are similar, but since $\overline{P_{1,i-2}P_{2,i-1}}$ is twice the length of $\overline{P_{1,i-2}A}$, $\overline{P_{1,i-2}B}$ is twice the length of $\overline{P_{1,i-2}P_{1,i-1}}$, and so \begin{flalign*}&\dsp \ell(\overline{P_{1,i-1}P_{1,i}}) > \ell(\overline{P_{1,i-1}B})= \ell(\overline{P_{1,i-2}P_{1,i-1}}).&\end{flalign*} 

Proposition \ref{Proposition:Prop3point2} implies the congruency of the angles $\angle P_{1,i}P_{2,i-1}P_{1,i-1}$ and $\angle P_{i,2}P_{2,i-1}P_{1,i}$.  Proposition \ref{Proposition:Anglemeasure} implies that  \begin{flalign*}&\dsp m(\angle P_{1,i}P_{2,i-1}P_{1,i-1}) = 90^\circ-\frac{i-2}{2}\theta^\circ.&\end{flalign*} Proposition \ref{Proposition:Prop3point2} implies that $\overline{P_{2,i-2}P_{1,i-1}}$ is perpendicular to $\overline{P_{1,i-2}P_{2,i-1}}$ and furthermore that $\overline{P_{2,i-2}P_{1,i-1}}$ bisects $\angle P_{2,i-1}P_{2,i-2}P_{1,i-2}$.  Therefore, $\angle P_{1,i-2}P_{2,i-1}P_{2,i-2}$ and $\angle P_{2,i-2}P_{1,i-2}P_{2,i-1}$ are congruent.  Proposition \ref{Proposition:Anglemeasure} implies that \begin{flalign*}&\dsp 2m(\angle P_{1,i-2}P_{2,i-1}P_{2,i-2}) + 180^\circ -(i-3)\theta^\circ = 180^\circ.&\end{flalign*}  Therefore, $m(\angle P_{1,i-2}P_{2, i-1}P_{2,i-2})$ is equal to $\frac{i-3}{2}\theta^\circ$ and so
\begin{flalign*}\dsp m(\angle P_{1,i}P_{2,i-1}P_{1,i-2}) &= m(\angle P_{1,i}P_{2,i-1}P_{1,i-1})+ m(\angle P_{1,i-1}P_{2,i-1}P_{1,i-2})\\&= 90^\circ-\tfrac{i-2}{2}\theta^\circ + \left((i-2)\theta^\circ - \tfrac{i-3}{2}\theta^\circ\right)\\ &= 90^\circ + \tfrac{\theta}{2}^\circ > 90^\circ.&\end{flalign*}

Consider the case where $i$ is $2$ (Figure \ref{o}).  As in the argument above, the key is to show that \begin{flalign*}&\dsp m(\angle P_{1,3}P_2P_1)>90^\circ.&\end{flalign*}

Proposition \ref{Proposition:Prop3point2} implies that $\overline{P_{1,2}P_{3,2}}$ is tangent to $\mathscr A$, hence perpendicular  to $\overline{P_{1,3}P_2}$.  It suffices to show that $m(\angle P_{1,2}P_2P_1)$ does not equal 0. Since \begin{flalign*}&\dsp m(\angle P_2OP_1) = \theta^\circ,&\end{flalign*} $\angle P_1P_{1,2}P_2$ has measure $(180 - \theta)^\circ$.  Since the angles $\angle P_{1,2}P_1P_2$ and $\angle P_{1,2}P_2P_1$ are congruent, both have measure $\frac{\theta}{2}^\circ$, which is greater than $0$.
\end{proof}

\begin{theorem}\label{Theorem:TheoremThree}
Suppose that $n$ is a natural number greater than two and $(P_1, \dots, P_{n+1})$ is a regular clockwise orientated partition of an arc $\mathscr A$. If $L_n$ and $L_m$  respectively denote the length of the line segment $\overline{P_1P_{1,n+1}}$ and the length of the line segment $\overline{P_1P_{m+1}}$, then \begin{flalign*}&\dsp m <n \quad \text{implies that}\quad nL_m < mL_n.&\end{flalign*} 
\end{theorem}

\begin{proof}
The lengths of the segments $\overline{P_{1,i}P_{1,i+1}}$ are increasing in $i$, so the average length of a segment of $L_n$ is larger than the average length of a segment of $L_m$.  Therefore, \begin{flalign*}&\dsp \tfrac{L_n}{n} > \tfrac{L_m}{m},&\end{flalign*}  implying the desired result.  
\end{proof}

\section{Circumference and Area are Intrinsic}

\begin{definition}
A \emph{circuit} is a finite sequence of points $\left(P_1, \dots, P_{n+1}\right)$ on $\C$ that is counterclockwise ordered and with the property that $P_1$ is equal to $P_{n+1}$. 
\end{definition}

Without loss in generality and in order to simplify the exposition, assume that the arc connecting adjacent points of a circuit is less than half of a circle. 

\begin{definition}
A \emph{refinement of a circuit $\mathcal S_1$} is a circuit $\mathcal S_2$ such that, as functions on finite sets, the range of $\mathcal S_1$ is a subset of the range of $\mathcal S_2$.  
\end{definition}

Denote by $S(\C)$ the set of all circuits on $\C$. Suppose that $\left(P_1,\dots, P_{n+1}\right)$ is in $S(\C)$.  The inscribed polygon $g\!\left(P_1,\dots,P_{n+1}\right)$ is the polygon whose vertices are the points in the circuit and the circumscribed polygon $G\!\left(P_1,\dots,P_{n+1}\right)$ is the polygon whose vertices are the points given by the intersections of the lines tangent to $\C$ at adjacent points of the circuit.  

\begin{definition}
For any polygon $P$, denote by $\pi(P)$ the perimeter of $P$. Let $\min\!\left(P_1,\dots,P_{n+1}\right)$ denote the minimum edge length of $g\!\left(P_1,\dots,P_{n+1}\right)$ and denote by $\max\!\left(P_1,\dots,P_{n+1}\right)$ the maximal edge length of $g\!\left(P_1,\dots,P_{n+1}\right)$.
\end{definition}

\medskip

\dwpicp

\medskip

\subsection{Approximation By Rational Circuits}

Let the length $\ell$ be less than 2.  Construct recursively a sequence of points and a corresponding piecewise linear path in the following way. Take $P_1$ to be a point on $\C$.  Let $P_n$ be the $n^{\rm th}$ term in a sequence of points on $\C$ so that if $P_i$ and $P_{i+1}$ are adjacent points in the sequence, then $\ell(\overline{P_iP_{i+1}})$ is equal to $\ell$ and the triangle $\triangle OP_iP_{i+1}$ is counterclockwise oriented. The circle $\C_{\ell,n}$ of radius $\ell$ with center $P_n$ intersects $\C$ in exactly two places.  Let $P_{n+1}$ be the first of these points counterclockwise from $P_n$.  The line segment $\overline{P_nP_{n+1}}$ is of length $\ell$.  If the sequence has the property that for some $k$, $P_k$ is equal to $P_1$, then $\ell$ is a \emph{rational length} and there is a first such $k$ that is the \emph{numerator of $\ell$}, denoted by ${\mathcal N}(\ell)$. The closed piecewise linear path $\Gamma(\ell)$ is the finite sequence of edges \begin{flalign*}&\dsp \Gamma(\ell) = \left(\overline{P_1P_2},\dots, \overline{P_{{\mathcal N}(\ell)-1}P_{{\mathcal N}(\ell)}}\right).&\end{flalign*}  Let $\overline{OP_1}$ be the line from the origin to the point $P_1$. Denote by ${\mathcal D}(\ell)$, the \emph{denominator of $\ell$},  the number of times the lines $\overline{P_iP_{i+1}}$ intersect the line segment $\overline{OP_1}$, where $i$ in $(1, {\mathcal{N}}(\ell)]$. The denominator counts how many times $\Gamma(\ell)$ wraps around the origin.  A length $\ell$ is \emph{integral} if ${\mathcal D}(\ell)$ is equal to 1, implying that $\ell$ is the edge length of a regular inscribed polygon.  A circuit is said to be \emph{integral} if all of its edge lengths are integral.  A \emph{regular circuit} is an integral circuit in which the distance between any two adjacent points in the circuit is the same.  The counterclockwise ordered set of intersections of any circumscribed regular polygon with $\C$ is a regular circuit.  Furthermore, any regular circuit is the set of intersections of a regular circumscribed polygon with $\C$.

\begin{proposition}\label{Proposition:Propositionfourone}
Let $\ell$ and $m$ be rational lengths. Let $L$ and $M$ denote the edge lengths of circumscribed edges corresponding to the inscribed edges $\ell$ and $m$ respectively. If $m$ is less than $\ell$, then \begin{flalign*}&\dsp\dfrac{{\mathcal N}(\ell)}{{\mathcal D}(\ell)} \ell < \dfrac{{\mathcal N}(m)}{{\mathcal D}(m)} m\quad {\rm and} \quad\dfrac{{\mathcal N}(\ell)}{{\mathcal D}(\ell)} L > \dfrac{{\mathcal N}(m)}{{\mathcal D}(m)} M.&\end{flalign*}
\end{proposition}

\begin{remark}
If $\ell$ and $m$ are integral lengths, then ${\mathcal D}(\ell)$ and ${\mathcal D}(m)$ are both equal to 1 and the proposition gives a comparison of the perimeters of the regular polygons of side lengths $\ell$ and $m$.  In the rational but non-integral case, it is not reasonable to compare the lengths of the curves $\Gamma(\ell)$ and $\Gamma(m)$ because of the dissimilar wrapping of the curves $\Gamma(\ell)$ and $\Gamma(m)$ around the circle.  Dividing the lengths of the above curves by their respective denominators gives a way to compare the average length of the curves on a single wrapping around the circle.  It is this idea that will prove important in the next proposition.
\end{remark}

\begin{proof}
Since length is invariant under rotation, assume that both paths $\Gamma(\ell)$ and $\Gamma(m)$ start with the same initial point.  All the points on the paths $\Gamma(\ell)$ and $\Gamma(m)$ are points on the regular polygon with ${\mathcal N}(\ell){\mathcal N}(m)$ sides containing the point $P_1$. Since $\ell$ may be realized as the edge length of a straight line segment traversing ${\mathcal{D}}(\ell)$ segments of a regular ${\mathcal N}(\ell)$-gon and $m$ may be realized as the edge length of a straight line segment traversing ${\mathcal D}(m)$ segments of a regular ${\mathcal N}(m)$-gon, $\ell$ may also be realized as the edge length of a straight line segment traversing ${\mathcal{D}}(\ell){\mathcal N}(m)$ segments of a regular ${\mathcal N}(\ell){\mathcal N}(m)$-gon and $m$ may also be realized as the edge length of a straight line segment traversing ${\mathcal D}(m){\mathcal N}(\ell)$ segments of a regular ${\mathcal N}(m)$-gon. Since $\ell$ is larger than $m$, \begin{flalign*}&\dsp {\mathcal{D}}(\ell){\mathcal N}(m) > {\mathcal D}(m){\mathcal N}(\ell).&\end{flalign*} By assumption, both lengths are less than a diameter.  Theorem \ref{Theorem:TheoremTwo} therefore implies that \begin{flalign*}&\dsp {\mathcal D}(\ell){\mathcal N}(m) m < {\mathcal D}(m){\mathcal N}(\ell) \ell,&\end{flalign*} and this proves the proposition.

The case of circumscribed polygons is similar except that instead of comparing the lengths $\ell$ and $m$, compare the lengths $L$ and $M$ and appeal to Theorem \ref{Theorem:TheoremThree} to obtain the reverse inequality.
\end{proof}

\begin{proposition}\label{Proposition:PropositionFourTwo}
If $(P_1, \dots, P_{k+1})$ and $(Q_1,\dots, Q_{n+1})$ are two rational circuits, then \begin{flalign*}&\dsp \min(P_1, \dots, P_{k+1}) > \max(Q_1,\dots, Q_{l+1})&\end{flalign*} implies that \begin{flalign*}&\dsp (1)\quad \pi (g(Q_1,\dots, Q_{l+1})) > \pi (g(P_1, \dots, P_{k+1}))& \end{flalign*} and \begin{flalign*}&\dsp (2)\quad  \pi (G(Q_1,\dots, Q_{l+1})) < \pi (G(P_1, \dots, P_{k+1})).&\end{flalign*}
\end{proposition}

\begin{proof}
Denote by $\ell_i$ and $m_i$ the lengths \[\ell_i = \ell(\overline{P_iP_{i+1}})\quad {\rm and} \quad m_j =
\ell(\overline{Q_jQ_{j+1}}).\]  Abuse notation and denote again by $\ell_i$ and $m_i$ the respective segments of length $\ell_i$ and $m_i$.  There are $k$ segments $\ell_1,\dots,
\ell_k$ for $g\!\left(P_1,\dots, P_{k+1}\right)$ and $n$ segments
$m_1,\dots,m_n$ for $g\!\left(Q_1,\dots,Q_{n+1}\right)$.  Let $\ell_1$ be the shortest of the segments of $g\!\left(P_1,\dots, P_{k+1}\right)$ and $m_n$ be the longest of the segments of $g\!\left(Q_1,\dots,Q_{n+1}\right)$.  Take
$\Lambda$ to be the product \[\Lambda = {\mathcal N}(\ell_1)\cdots{\mathcal N}(\ell_k){\mathcal
N}(m_1)\cdots{\mathcal N}(m_n)\] and take $\Lambda$ copies of both polygons.  Compute
the respective perimeters of these polygons to obtain

\begin{flalign*}
\dsp\Lambda \pi(g(P_1, \dots, P_{k+1})) &=
\Lambda(\ell_1 + \ell_2 + \dots + \ell_k)\\ &=
\frac{\Lambda}{{\mathcal N}(\ell_1)} {\mathcal N}(\ell_1) \ell_1  + \dots +
\frac{\Lambda}{{\mathcal N}(\ell_k)} {\mathcal N}(\ell_k) \ell_k \\ & =
{\mathcal D}(\ell_1)\frac{\Lambda}{{\mathcal N}(\ell_1)} \frac{{\mathcal
N}(\ell_1)}{{\mathcal D}(\ell_1)} \ell_1 + \cdots + {\mathcal
D}(\ell_k)\frac{\Lambda}{{\mathcal N}(\ell_k)} \frac{{\mathcal
N}(\ell_k)}{{\mathcal D}(\ell_k)} \ell_k \\%
& <
{\mathcal D}(\ell_1)\frac{\Lambda}{{\mathcal N}(\ell_1)} \frac{{\mathcal
N}(\ell_1)}{{\mathcal D}(\ell_1)} \ell_1 + \cdots + {\mathcal
D}(\ell_k)\frac{\Lambda}{{\mathcal N}(\ell_k)} \frac{{\mathcal
N}(\ell_1)}{{\mathcal D}(\ell_1)} \ell_1\\
 &=\left({\mathcal D}(\ell_1)\frac{\Lambda}{{\mathcal N}(\ell_1)} + \cdots +
{\mathcal D}(\ell_k)\frac{\Lambda}{{\mathcal N}(\ell_k)} \right)
\frac{{\mathcal N}(\ell_1)}{{\mathcal D}(\ell_1)}\ell_1=
\Lambda \frac{{\mathcal N}(\ell_1)}{{\mathcal D}(\ell_1)} \ell_1.&
\end{flalign*}

The inequality follows from Theorem \ref{Theorem:TheoremTwo} since $\ell_1$ is the minimum length of the segments of $g(P_1, \dots, P_{k+1})$.  The ultimate equality follows from the fact that \begin{flalign*}
&\dsp \sum_{i=1}^k {\mathcal D}(\ell_i)\frac{\Lambda}{{\mathcal N}(\ell_i)} = \Lambda.&\end{flalign*}  To justify the last assertion, note that given $\Lambda$ copies of the polygon $g(P_1, \dots, P_{k+1})$, the segments in all of these copies form a path that wraps around $\C$ exactly $\Lambda$ times.  However, this collection of polynomial contains $\Lambda$ copies of each segment $\ell_i$, hence $\frac{\Lambda}{{\mathcal N}(\ell_i)}$ copies of each path $\Gamma(\ell_i)$ which wraps around $\C$ exactly ${\mathcal D}(\ell_i)$ times. Summing each of these $\frac{\Lambda}{{\mathcal N}(\ell_i)}{\mathcal D}(\ell_i)$ wrappings gives the total number of wrappings, $\Lambda$, implying the above formula for the sum.

Similarly, use the reverse inequality and make use of Proposition \ref{Proposition:Propositionfourone} to obtain the inequality
\begin{flalign*}
&\dsp \Lambda\pi(g(Q_1, \dots, Q_{n+1})) > \Lambda \frac{{\mathcal N}(m_n)}{{\mathcal D}(m_n)} m_n > \Lambda \frac{{\mathcal N}(\ell_1)}{{\mathcal D}(\ell_1)} \ell_1 > \Lambda\pi(g(P_1, \dots, P_{k+1})).&\end{flalign*} Dividing both sides of the inequality by $\Lambda$ finishes the proof for inscribed polygons.

The proof of the result for circumscribed polygons is done in the same way but appeals to the appropriate inequality in Proposition~\ref{Proposition:Propositionfourone}, reversing the inequalities analogous to those given above.
\end{proof}

\begin{definition}
Following Archimedes, define $2\pi$ to be equal to $p_3$ (as defined in Proposition~\ref{pPaAn}). 
\end{definition}

Denote by ${\mathds Q}({\C})$ the set of all rational circuits on $\C$ and suppose that $\P$ is in ${\mathds Q}({\C})$. Denote  the maximal edge length of $\P$, the \emph{mesh} of $\P$, by $\mu(\P)$.    

\begin{theorem}\label{Theorem:TheoremFour}
For any positive real number $\varepsilon$ there is a positive real number $\delta$ such that \begin{flalign*}&\dsp \mu(\P)  < \delta\quad \text{implies that}\quad 0 < 2\pi -
\pi(g(\P)) < \varepsilon \quad  \text{and}\quad
0 < \pi(G(\P))-2\pi < \varepsilon.&\end{flalign*}
\end{theorem}

\begin{proof}
Let $\varepsilon$ be a positive real number.  Both $\pi(\g{m}{3})$ and $\pi(\G{m}{3})$ tend to $2\pi$ as $m$ tends to infinity, and so as long as $m^\prime$ is a large enough natural number, \begin{flalign}\label{Theorem:TheoremFour:a}&\dsp 0 < 2\pi - \pi(\g{m^\prime}{3}) < \varepsilon \quad {\rm and} \quad 0 < \pi(\G{m^\prime}{3}) - 2\pi < \varepsilon.&\end{flalign} Let $\mathcal P$ be a rational circuit.  Take $\delta$ to be a positive real number less than the side length of $\g{m^\prime}{3}$, so that the maximum possible side length of $G(\mathcal P)$ is less than the side length of $\G{m^\prime}{3}$ and the maximum possible side length of $g(\mathcal P)$ is less than the edge length of $\g{m^\prime}{3}$.  Since $\mu(\mathcal P)$ is less than $\delta$, Proposition~\ref{Proposition:PropositionFourTwo} implies that \begin{flalign}\label{Theorem:TheoremFour:b}&\dsp \pi(\g{m^\prime}{3})<\pi(g(\mathcal P)) \quad {\rm and}\quad \pi(\G{m^\prime}{3})>\pi(G(\mathcal P)).&\end{flalign}  Take $m^{\prime\prime}$ to be a large enough natural number so that the edge length of $\pi(\g{m^{\prime\prime}}{3})$ is smaller than the smallest edge length of $g(\mathcal P)$ and the edge length of $\pi(\G{m^{\prime\prime}}{3})$ is smaller than the smallest edge length of $G(\mathcal P)$.  Such a choice of $m^{\prime\prime}$ is possible because the side lengths of both $\pi(\g{m}{3})$ and $\pi(\G{m}{3})$ tend to zero as $m$ tends to infinity.  Proposition~\ref{Proposition:PropositionFourTwo} implies that \begin{flalign}\label{Theorem:TheoremFour:c}&\dsp \pi(g(\mathcal P)) < \pi(\g{m^{\prime\prime}}{3}) < 2\pi \quad {\rm and}\quad \pi(G(\mathcal P)) > \pi(\G{m^{\prime\prime}}{3})> 2\pi.&\end{flalign}  Inequalities \eqref{Theorem:TheoremFour:a}, \eqref{Theorem:TheoremFour:b}, and  \eqref{Theorem:TheoremFour:c} together imply the theorem.
\end{proof}

Denote by $\pi(n)$ the perimeter of a regular $n$-gon inscribed in $\C$ and denote by $\Pi(n)$ the perimeter of a regular circumscribed $n$-gon.  

\begin{corollary}[Corollary to Proposition \ref{Proposition:PropositionFourTwo}]
If $n$ is greater than $m$ and $m$ is greater than 2, then \begin{flalign*}&\dsp (1)\quad  \pi(n) > \pi(m)\quad \text{and}\quad (2)\quad  \Pi(n) < \Pi(m).&\end{flalign*}
\end{corollary}

\begin{corollary}[Corollary to Theorem \ref{Theorem:TheoremFour}]
For any natural numbers $n$ and $m$ larger than 2, 
\begin{flalign*}&\dsp p_m = p_n, \quad P_m = P_n, \quad a_m = a_n, \quad {\rm and}\quad A_m = A_n,&\end{flalign*} where the notation follows Proposition~\ref{pPaAn}.  
\end{corollary}

\begin{remark}
The next section presents results somewhat more general than those of the above theorem. Until this point we have only used the convergence of bounded monotone sequences, the axioms of Euclidean geometry, and the properties of the natural numbers. In this next section, we use a cardinality argument to show that not all edge lengths are rational lengths and so this section relies on a consideration that is highly unlikely to have been considered by a contemporary of Archimedes, but is necessary from a modern perspective.
\end{remark}

\subsection{Approximation by General Circuits}

Each rational length corresponds to a pair of two integers, a numerator and a denominator, making the set of rational lengths a countably infinite set.  However, there are uncountably many possible edge lengths and so this necessitates a study of circuits that are not necessarily rational.

\begin{proposition}\label{Proposition:PropositionFourFour}
Given any circuit $(R_1, \dots, R_{n+1})$ and a positive real number $\varepsilon$, there is a rational circuit $(P_1,\dots, P_{n+1})$ such that \begin{flalign*}&\dsp |\pi(g(R_1, \dots, R_{n+1})) - \pi(g(P_1,\dots, P_{n+1}))|<\varepsilon.&\end{flalign*}
\end{proposition}

\begin{proof}
Suppose that a circuit $\mathcal R$ has $n$ distinct points and thus corresponds to a vertex set of an inscribed polygon with $n$ sides. Let $\varepsilon$ be a positive real number.  There is a natural number $m$ such that the edge length $\ell$ of $\g{m}{3}$ is smaller than $\frac{\varepsilon}{2n}$.  Take $R_1$ to be a point of $\g{m}{3}$.  To each point $R_i$ of $\mathcal R$, take $P_i$ to be $R_i$ if $R_i$ is a vertex of $\g{m}{3}$.  Otherwise, take $P_i$ to be the first vertex of $\g{m}{3}$ that is counterclockwise from $R_i$. With such a choice of $P_i$, the length $\ell({\overline{P_iR_i}})$ is less than $\ell$.  The circuit $\P$ that is equal to $(P_1, \dots, P_{n+1})$ is a rational circuit.  Furthermore, \begin{flalign*}&\dsp \left|\overline{\ell(P_iP_{i+1}}) -
\ell(\overline{R_iR_{i+1}})\right| < 2\ell < \frac{\varepsilon}{n}&\end{flalign*} for each natural number $i$ in $[1, k+1]$.  Therefore, \begin{flalign*}&\dsp |\pi(g(\P)) - \pi(g({\mathcal R}))| < \varepsilon.&\end{flalign*}%
\end{proof}

\begin{corollary}
For any circuits $(P_1, \dots, P_{k+1})$ and $(Q_1,\dots, Q_{l+1})$ on $\C$, \begin{flalign*}&\dsp \min(P_1, \dots, P_{k+1}) > \max(Q_1,\dots, Q_{l+1}) \implies \pi(Q_1,\dots, Q_{l+1}) \geq \pi(P_1, \dots, P_{k+1}).&\end{flalign*}
\end{corollary}

\begin{proof}
Suppose that $\mathcal P$ is the circuit $(P_1, \dots, P_{k+1})$ and that $\mathcal Q$ is the circuit $(Q_1,\dots, Q_{l+1})$.  Suppose further that $\min(P_1, \dots, P_{k+1})$ is greater than $\max(Q_1,\dots, Q_{l+1})$.  For any positive real number $\varepsilon$, there are rational circuits $\hat{\mathcal P}$ and $\hat{\mathcal Q}$ such that \begin{flalign*}&\dsp |\pi(g(\hat{\mathcal P})) - \pi(g({\mathcal P}))| < \frac{\varepsilon}{2} \quad {\rm and} \quad |\pi(g(\hat{\mathcal Q})) - \pi(g({\mathcal Q}))| < \frac{\varepsilon}{2}&\end{flalign*} and furthermore such that \begin{flalign*}&\dsp \min(\hat{P}_1, \dots, \hat{P}_{k+1}) > \max(\hat{Q}_1,\dots, \hat{Q}_{l+1}).&\end{flalign*}  We leave the straightforward details of the proof to the reader.  Proposition~\ref{Proposition:PropositionFourFour} implies that $\pi(\hat{\mathcal Q})$ is greater than $\pi(\hat{\mathcal P})$, and so \begin{flalign*}&\dsp \pi({\mathcal Q}) + \varepsilon > \pi({\mathcal P}).&\end{flalign*}  Since the above inequality holds for any positive $\varepsilon$, $\pi({\mathcal Q})$ is greater than or equal to $\pi({\mathcal P}).$
\end{proof}

Given any two points on $\C$, there is an $m$ large enough so that the arc between the two points contains two vertices of $\g{m}{3}$.  The following lemma follows from this fact and the fact that rotations preserve the ordering of points on $\C$.

\begin{lemma}
Given points $P$, $Q$, and $R$ on $\C$ in counterclockwise order, there is a point $S$ between $Q$ and $R$ such that $\ell(\overline{PS})$ is a rational length.
\end{lemma}%
Establish the following notation for the statements and proofs of Lemma~\ref{Lemma:LemBij} and Lemma~\ref{Lemma:LemBij2}.  Suppose that $A$, $B$ and $C$ are points on $\C$ in counterclockwise order and the arc from $A$ to $C$ is less than half of $\C$.  For any points $X$ and $Y$ on $\C$, denote by $P_{X, Y}$ the intersection of the lines tangent to $\C$ at $X$ and $Y$ \lpar Figure~\ref{r} \rpar.

\bigskip

\dwpicr

\bigskip

\begin{lemma}\label{Lemma:LemBij}
There is an order preserving bijection from the points on the line segment $\overline{P_{A, B}P_{A, C}}$ and the points on the arc between $B$ and $C$.
\end{lemma}

\begin{proof}
Suppose $P$ is a point on $\overline{P_{A,C}P_{B,C}}$.  Let $\C_P$ be the circle of radius $\ell(\overline{CP})$ centered at $P$.  The circle $\C_p$ intersects $\C$ at precisely two points, at the point $C$ and at a point $Q$ that lies on the arc from $A$ to $B$.  Define the circles $\C_{P_{B,C}}$ and $\C_{P_{A,C}}$ in the same way as $\C_P$.  The circles $\C_{P_{B,C}}$, $\C_P$, and $\C_{P_{A,C}}$ all intersect at $C$ and, since they have different radii, they can meet at no more than two points.  Since $\overline{OC}$ is tangent to all three circles, the three circles meet only at $C$.  The radius of $\C_{P_{A,C}}$ is greater than the radius of $\C_{P}$, which is greater than the radius of $\C_{P_{B,C}}$, therefore \begin{flalign*}&\dsp \ell(\overline{BC}) <  \ell(\overline{QC})  < \ell(\overline{AC}).&\end{flalign*}  Since all three points lie counterclockwise from $A$ on an arc less that half of a circle, $Q$ is clockwise from $B$ and counterclockwise from $A$.  

Suppose that $Q$ is a point between $A$ and $B$ on the arc from $A$ to $B$, that $B$ is clockwise from $C$, and that the arc from $A$ to $C$ is less than half of a circle.  Let $L$ be the line tangent to $\C$ at $Q$.  Since $\angle QOC$ is less than a straight line, $L$ intersects the line tangent to $\C$ at $C$ at a point $P$ and the intersection occurs on the same side of $C$ as $P_{A,C}$ and $P_{B,C}$.  The point $Q$ is between $A$ and $B$, and so \begin{flalign*}&\dsp \ell(\overline{BC}) <  \ell(\overline{QC})  < \ell(\overline{AC}),\quad \text{hence}\quad m(\angle{BOC}) <  m(\angle{QOC})  < m(\angle{AOC}).&\end{flalign*}   The line segments $\overline{BO}$, $\overline{QO}$, and $\overline{AO}$ are all radii, implying the equality of $\ell(\overline{BO})$, $\ell(\overline{QO})$, and $\ell(\overline{CO})$ and, therefore, the inequalities \begin{flalign*}&\dsp \ell(\overline{CP_{B,C}}) <  \ell(\overline{CP})  < \ell(\overline{CP_{A,C}}).&\end{flalign*}  The point $P$ therefore lies on the line segment $\overline{P_{A,C}P_{B,C}}$.  Note that Lemma \ref{Lemma:LemmaMain} implies that $\overline{PQ}$ and $\overline{CP}$ are congruent and so the map we initially constructed will map $P$ to the point $Q$.

Let $\phi$ be the map taking points on $\overline{P_{A,C}P_{B,C}}$ to points on the arc from $A$ to $B$ and let $\psi$ be the map taking points on the arc from $A$ and $B$ to point on $\overline{P_{A,C}P_{B,C}}$.  These functions are inverses of each other, and so both are invertible, hence bijective.
\end{proof}

\begin{lemma}\label{Lemma:LemBij2}
Let $\varepsilon$ be a positive real number.  There is a $Q$ on the arc from $A$ to $B$ such that $Q$ is not equal to $B$ and \begin{flalign}\label{lastineq}&\dsp \left|\ell(\overline{AP_{A,B}}) +  \ell(\overline{P_{A,B}B}) +\ell(\overline{BP_{B,C}}) + \ell(\overline{P_{B,C}C})\right.&\notag\\&\dsp\qquad\qquad\qquad\qquad\left. - \left(\ell(\overline{AP_{A,Q}}) +  \ell(\overline{P_{A,Q}Q}) + \ell(\overline{QP_{C,Q}}) + \ell(\overline{P_{C,Q}C})\right)\right| < \varepsilon.&\end{flalign} Furthermore, any point $Q^\prime$ on the arc from $Q$ to $B$ will also satisfy the above estimate where $Q$ is replaced by $Q^\prime$.
\end{lemma}

\begin{proof}
Pick a point $z$ on $\overline{AP_{A,B}}$ such that $\ell(\overline{zP_{A,B}})$ is less than $\frac{\varepsilon}{4}$.   Lemma~\ref{Lemma:LemBij} guarantees the existence of a point $Z$ on the arc between $A$ and $B$ such that $z$ is the point of intersection of the line tangent to $\C$ at $Z$ and the line segment $\overline{AP_{A,B}}$.  Pick a point $q$ on $\overline{zP_{A,B}}$ such that $\ell(\overline{qP_{A,B}})$ is less than $\frac{\varepsilon}{4}$.   Lemma~\ref{Lemma:LemBij} guarantees the existence of a point $Q$ on the arc from $Z$ to $B$ such that $q$ is the point of intersection of the line tangent to $\C$ at $Q$ and the line segment $\overline{AP_{A,B}}$.  Using the established notation, $q$ is the point $P_{A,Q}$ and, furthermore, $Q$ will satisfy the estimate \eqref{lastineq}.  Since the bijection given in Lemma~\ref{Lemma:LemBij} is order preserving, if $Q^\prime$ is any point not equal to $B$ in the arc from $Q$ to $B$, then the point $Q^\prime$ will also satisfy the estimate \eqref{lastineq}.
\end{proof}

\begin{proposition}\label{Proposition:PropositionFourFive}
Suppose that $(P_1, \dots, P_{k+1})$ is a circuit with the property that the arc between any adjacent points is less than a fourth of a circle.  Suppose $\varepsilon$ is a positive real number. There is a rational circuit $(Q_1,\dots, Q_{l+1})$ such that \begin{flalign*}&\dsp |\pi(G(P_1, \dots, P_{k+1})) - \pi(G(Q_1,\dots, Q_{l+1}))|<\varepsilon.&\end{flalign*}
\end{proposition}

\begin{proof}
Set $Q_1$ equal to $P_1$.  Suppose that $(Q_1, Q_2, \dots, Q_{i}, P_{i+1}, \dots, P_{k+1})$ is a sequence of points where each $Q_i$ lies on the arc between $P_i$ and $P_{i+1}$, each $Q_i$ lies on the same regular polygon, each edge is less than a quarter of the circle, and \begin{flalign*}&\dsp \left|\pi(G(P_1, P_2, \dots, P_{i}, P_{i+1}, \dots, P_{k+1}))  \right. &\\ &\dsp \qquad\qquad\qquad\qquad\left.- \pi(G(Q_1, Q_2, \dots, Q_{i}, P_{i+1}, \dots, P_{k+1}))\right|<\varepsilon.&\end{flalign*}  Lemma~\ref{Lemma:LemBij2} implies that there is a $Q_{i+1}$ not equal to $P_{i+1}$ but on the arc between $Q_{i+1}$ and $P_{i+2}$ such that for any $Q$ on the arc between $Q_{i+1}$ and $P_{i+2}$, \begin{flalign*}&\dsp \left|\pi(G(P_1, P_2, \dots, P_{i}, P_{i+1}, \dots, P_{k+1}))  \right. &\\ &\dsp \qquad\qquad\qquad\qquad\left.- \pi(G(Q_1, Q_2, \dots, Q_{i}, Q, P_{i+2}, \dots, P_{k+1}))\right|<\varepsilon.&\end{flalign*}  Take $Q$ to be a vertex of a regular polygon on which $Q_1$ through $Q_i$ lie by choosing the regular polygon to be a sufficiently fine refinement of the regular polygon on which $Q_1$ through $Q_i$ are vertices.  Recursively construct in this way a polygon $G(Q_1, \dots, Q_{k+1})$, with $Q_1$ and $Q_{k+1}$ both equal to $P_{1}$, which approximates $G(P_1, \dots, P_{k+1})$ in the sense that \begin{flalign*}&\dsp \left|\pi(G(P_1, \dots, P_{k+1})) - \pi(G(Q_1,\dots, Q_{k+1}))\right|<\varepsilon.&\end{flalign*}
\end{proof}

\begin{corollary}
If $(P_1, \dots, P_{k+1})$ and $(Q_1,\dots, Q_{l+1})$ are two arbitrary sequences of points on $C$
then \begin{flalign*}&\dsp \min(P_1, \dots, P_{k+1}) > \max(Q_1,\dots, Q_{l+1})&\end{flalign*} implies that \begin{flalign*}&\dsp \pi(G(Q_1,\dots, Q_{l+1})) \leq \pi(G(P_1, \dots, P_{k+1})).&\end{flalign*}
\end{corollary}

\begin{proof}
The above corollary is proved exactly as the first corollary to Proposition~\ref{Proposition:PropositionFourFour} is proved above except that the inequalities are reversed because they are reversed in Proposition~\ref{Proposition:PropositionFourTwo}.
\end{proof}

\begin{theorem}\label{Theorem:TheoremFive}
Suppose that $\P$ is a general partition of $\C$.  For any positive real number $\varepsilon$ there is a positive real number $\delta$ such that \begin{flalign*}&\dsp \mu(\P) < \delta\quad \text{implies that} \quad 0 < 2\pi - \pi(g(\P)) < \varepsilon \quad \text{and} \quad 0 < \pi(G(\P))-2\pi < \varepsilon.&\end{flalign*}
\end{theorem}

\begin{proof}
The theorem is proved in the same way as Theorem~\ref{Theorem:TheoremFour} but by appealing to Proposition~\ref{Proposition:PropositionFourTwo} and the corollary to Proposition~\ref{Proposition:PropositionFourFive} rather than Proposition~\ref{Proposition:PropositionFourTwo} and the corollary to Proposition~\ref{Proposition:PropositionFourFour}.
\end{proof}

Let $\P$ be a general partition of $\C$.  Denote by $\alpha(g(\P))$ the area of the inscribed polygon $g(\P)$ and $\alpha(G(\P))$ the area of the circumscribed polygon $G(\P)$.

\begin{theorem}
For any positive real number $\varepsilon$ there is a  positive real number $\delta$ such that  \begin{flalign*}&\dsp \mu(\P) < \delta\quad \text{implies that} \quad 0 < \pi - \alpha(g(\P)) < \varepsilon \quad \text{and} \quad 0 < \alpha(G(\P))-\pi < \varepsilon.&\end{flalign*}
\end{theorem}

\begin{proof}
Let $\varepsilon$ be a positive real number.  Theorem~\ref{Theorem:TheoremFive} implies that there is a positive real number $\delta$ so that if $\P$ is a circuit $(P_1, \dots, P_{n+1})$ with mesh less than $\delta$, then \begin{flalign}\label{5:ineq:Gpi}&\dsp  0 < 2\pi - \pi(g(\P)) < \varepsilon \quad \text{and} \quad 0 < \pi(G(\P))-2\pi < \varepsilon.&\end{flalign}  Restrict $\delta$ so that $\frac{\pi}{4}\delta^2$ is less than $\varepsilon$. The equality \begin{flalign*}&\dsp \alpha(G(\P)) = \dfrac{1}{2}\pi(G(\P))&\end{flalign*} together with \eqref{5:ineq:Gpi} implies that \begin{flalign*}&\dsp 0 < \alpha(G(\P))-\pi = \dfrac{1}{2}\pi(G(\P)) - \pi < \pi + \frac{\varepsilon}{2} - \pi = \frac{\varepsilon}{2}.&\end{flalign*}  Denote by $\ell_i$ the length of the segment $\overline{P_iP_{i+1}}$. Heron's Theorem implies that \begin{flalign*}&\dsp \alpha(g(\P)) = \sum_{i=1}^n \dfrac{\ell_i}{2}\sqrt{1 - {\dfrac{\ell_i^2}{4}}}.&\end{flalign*}   For each $i$, \begin{flalign*}&\dsp \dfrac{\ell_i}{2}\left(1 - \dfrac{\mu(\P)^2}{4}\right) \leq \dfrac{\ell_i}{2}\left(1 - \dfrac{\ell_i^2}{4}\right) \leq \dfrac{\ell_i}{2}\sqrt{1 - {\dfrac{\ell_i^2}{4}}} \leq \dfrac{\ell_i}{2}.&\end{flalign*}  Since the sum of all of the $\ell_i$'s is the perimeter of an inscribed polygon and therefore less than $2\pi$, the bound on $\delta$ implies that  \begin{flalign*}&\dsp \dfrac{\pi(g(\P))}{2} - \frac{\varepsilon}{2}  <  \dfrac{\pi(g(\P))}{2} -  \dfrac{\mu(\P)^2}{8}\sum_{i=1}^n \ell_i \leq \alpha(g(\P)) \leq \dfrac{\pi(g(\P))}{2}.&\end{flalign*}  This proves the theorem.
\end{proof}

\section{Application to the Circular Trigonometric Functions}

\subsection{Arclength of an Arc and Area of a Sector}

Suppose that $A$ and $B$ are points on $\C$ and that $\A$ is the counterclockwise oriented arc from $A$ to $B$.  If the length $\ell(\overline{AB})$ is a rational length, then for some natural number $m$, there is a regular inscribed $n$-gon $\g{0}{n}$ such that $\A$ has on it exactly $k+1$ vertices of $\g{0}{n}$ and both $A$ and $B$ are vertices of $\g{0}{n}$.  Denote by $\ell_n$ the edge length of $\g{0}{n}$ to obtain the equality \begin{flalign*}&\dsp k\ell_n = \frac{k}{n}\pi(\g{0}{n}).&\end{flalign*}  The inscribed polygon $\g{m}{n}$ will have $2^mn+1$ vertices on $\A$.  Denote by $\pi_{m,n}(\A)$ the approximation of the arclength of $\A$ given by  \begin{flalign*}&\dsp\pi_{m,n}(\A) = 2^mk\ell_n = \frac{2^mk}{2^mn}\pi(\g{m}{n}) = \frac{k}{n}\pi(\g{m}{n}) \to \frac{2\pi k}{n}\quad \text{as}\quad m\to \infty.&\end{flalign*} This approximation of the length of an arc of $\C$ depends only on the fraction of $\C$ that $\A$ represents and so defines the length of an arc for any arc whose endpoints form a rational length.  Define similarly the area of a sector of $\C$, the area bounded by the lines $\overline{OA}$ and $\overline{OB}$ and the counterclockwise oriented arc $\A$, so that  \begin{flalign*}&\dsp\alpha_{m,n}(\A) = \frac{k}{n}\alpha(\g{m}{n}) \to \frac{\pi k}{n}\quad \text{as}\quad m\to \infty.&\end{flalign*}

If the length $\ell(\overline{AB})$ is not a rational length, then there is a sequence of rational lengths that approximate it.  In particular, take $A$ to be a point on $\g{0}{n}$ for some $n$ where $n$ is a large enough natural number so that at least one vertex of $\g{0}{n}$ lies on $\A$ between $A$ and $B$.  For each $m$, choose the vertex $B_m$ of $\g{m}{n}$ that is closest to $B$.  The counterclockwise oriented arc $\A_m$ from $A$ to $B_m$ will have rational length $\ell(\overline{AB_m})$ and will be a fraction $\phi(AB_m)$ of a circle, where $\phi(AB_m)$ is an increasing bounded above sequence with \begin{flalign*}&\dsp\lim_{m\to \infty} \phi(AB_m) = \phi(AB),&\end{flalign*} where $\phi(AB)$ is an irrational number.  Obtain as $m$ tends to infinity the limits \begin{flalign*}&\dsp\pi_{m,n}(\A_m) = 2\pi \phi(AB_m) \to 2\pi \phi(AB) \quad {\rm and}\quad \alpha_{m,n}(\A_m) = \pi\phi(AB_m) \to \pi\phi(AB_m),&\end{flalign*} where the limits are independent of the sequence $(B_m)$ that approximates $B$.  

\medskip

Using the notation above and given any counterclockwise arc $\A$ from a point $X$ on $\C$ to a point $Y$ on $\C$, define the angle measure $\theta(\A)$ by  \begin{flalign*}&\dsp\theta(\A) = 2\pi\phi(XY).&\end{flalign*} The length $\theta$ is the length of the arc $\A$.  The area, $\alpha(\A)$ of the sector bounded by $\overline{OX}$, $\overline{OY}$, and the counterclockwise oriented arc $\A$ is given by \begin{flalign*}&\dsp\alpha(\A) = \frac{\theta(\A)}{2}.&\end{flalign*}

\subsection{Limit of the Sine Function}

The argument frequently given by authors of calculus texts for calculating the limit \eqref{limsinxoverx} is perfectly valid, although it does require some explanation since we have not yet introduced a coordinate system nor defined the trigonometric functions.  View the unit circle as the subset $\C$ of the plane given by \begin{flalign*}&\dsp\C = \{(x,y)\colon x^2 + y^2 = 1\}.&\end{flalign*}  Define the trigonometric functions, as is customary, as functions of the argument $\theta$ in $[0, 2\pi)$, so that the point $(\cos(\theta), \sin(\theta))$ is the point on the unit circle so that the counterclockwise oriented arc $\A$ from $(1,0)$ to $(\cos(\theta), \sin(\theta))$ has length $\theta$.  The area of the sector defined by $\A$ is $\frac{\theta}{2}$. Take $\theta$ to be less that $\frac{\pi}{2}$.  Bound the area of the sector above and below respectively by the areas of the triangles \[\triangle (0,0)(1,0)\Big(1, \frac{\sin(\theta)}{\cos(\theta)}\Big)\quad {\rm and}\quad \triangle(0,0)(\cos(\theta),0)(\cos(\theta),\sin(\theta))\] to obtain the inequality \begin{flalign*}&\dsp \frac{1}{2}\cos(\theta)\sin(\theta) \leq \frac{1}{2}\theta\leq \frac{1}{2}\frac{\sin(\theta)}{\cos(\theta)} \quad \text{implying that}\quad 1 \leq \frac{\theta}{\sin(\theta)} \leq \frac{1}{\cos(\theta)}.&\end{flalign*} Take limits and use the sandwich theorem to obtain the limit \eqref{limsinxoverx} as a one sided limit.  If $(x,y)$ is a point in the first quadrant so that the signed length of the arc from $(1,0)$ to $(x,y)$ is $\theta$, then define the signed length of the arc from $(1,0)$ to $(x,-y)$ to be $-\theta$.  Given this signed argument, the sine function is an odd function of $\theta$ and the cosine function is even.  Extending the definitions of the trigonometric functions in this way to negative signed arc lengths permits the limit \eqref{limsinxoverx} to be viewed as a two sided limit and the oddness of the sine function implies this two sided limit.  While this is a standard approach to calculating the limit in many calculus texts, \cite{Ste} for example, the approach is non-rigorous because these texts omit a discussion of the equivalence of $\pi$ defined as half the circumference of a unit circle and as the area of a unit circle and so as well the relationship between the length of an arc and the area of the sector that the arc defines.  Our current discussion closely follows Archimedes and remedies this shortcoming.

\end{document}